\documentclass[12pt,reqno]{amsart}
\usepackage{amssymb,amsmath,amsthm}
\usepackage{color}
\usepackage{pdflscape}
\usepackage{longtable}
\usepackage{enumerate}
\usepackage{tikz}
\usepackage{float}
\usepackage[colorlinks=true,
linkcolor=blue,
urlcolor=blue,
citecolor=blue]{hyperref}
\usepackage[paper=portrait,pagesize]{typearea}
\usepackage{pdflscape}
\usepackage{amsaddr}
\usepackage[margin=1.9cm]{geometry}
\newtheorem{theorem}{Theorem}[section]
\newtheorem{lemma}[theorem]{Lemma}
\theoremstyle{definition}

\newtheorem{proposition}[theorem]{Proposition}
\newtheorem{conjecture}[theorem]{Conjecture}
\newtheorem{corollary}[theorem]{Corollary}

\newtheorem{remark}[theorem]{Remark}

\newtheorem*{notation}{Notation}

\newcommand{\itbf}[1]{{\bf{{\emph{{#1}}}}}}
\newcommand{\Aut}[1]{\operatorname{Aut}(#1)}
\newcommand{\End}[2]{\operatorname{End}_{#1}(#2)}
\newcommand{\sym}[1]{\operatorname{Sym}(#1)}
\newcommand{\alt}[1]{\operatorname{Alt}(#1)}

\begin{document}
	
	\title[The Terwilliger algebra of $\sym{7}$]{On the Terwilliger algebra of the group association scheme of the symmetric group $\sym{7}$}
	
	\author[A.~Herman, R.~Maleki]{Allen Herman$^*$ {and} Roghayeh Maleki}\thanks{$^*$ The first author's research is supported by a grant from NSERC}
	\address{Department of Mathematics and Statistics, University of Regina,\\ Regina, SK S4S 0A2, Canada}\email{Allen.Herman@uregina.ca, RMaleki@uregina.ca}
	
	\author[A.~S.~Razafimahatratra]{Andriaherimanana Sarobidy Razafimahatratra$^\dag$}\thanks{$^\dag$ The third author gratefully acknowledge that this research was supported by the Fields Institute for Research in Mathematical Sciences. } 
	
	\address{Fields Institute for Research in Mathematical Sciences,\\ 222 College St., Toronto, ON M5T 3J1, Canada}\email{sarobidy@phystech.edu}
	
	\date{\today}
	
	\begin{abstract} 
	Terwilliger algebras are finite-dimensional semisimple algebras that were first 
introduced by Paul Terwilliger in 1992 in studies of association schemes and distance-regular graphs. The Terwilliger algebras of the conjugacy class association schemes of the symmetric groups $\sym{n}$, for $3\leq n \leq 6$, have been studied and completely determined. The case for $\sym{7}$ is computationally much more difficult and has a potential application to find the size of the largest permutation codes of $\sym{7}$ with a minimal distance of at least $4$. In this paper, the dimension, the Wedderburn decomposition, and the block dimension decomposition of the Terwilliger algebra of the conjugacy class scheme of the group $\sym{7}$  are determined.
	\end{abstract}
	
	\maketitle


	\tableofcontents
	
	\section{Introduction}
	Association schemes are central objects in algebraic combinatorics, which were introduced in the 1950s by Bose and Shimamoto. Paul Terwilliger introduced subconstituent algebras (now known as Terwilliger algebras) of association schemes in 1992 in the series of papers \cite{terwilliger1992subconstituent,terwilliger1993subconstituent2,terwilliger1993subconstituent3} as a means to study association schemes.  Many papers on Terwilliger algebras of various association schemes have appeared in the literature since then, see for instance \cite{BannaiMunemasa95,bhattacharyya2010terwilliger,caughman1999terwilliger,go2002terwilliger,herman2024terwilliger,levstein2006terwilliger,schrijver2005new,tan2019terwilliger}. These results mostly focus on the Terwilliger algebras of association schemes that are $P$-polynomial or $Q$-polynomial, and certain highly symmetric graphs arising from permutation group actions. 

   One important family of association schemes that arise from abstract groups is the conjugacy class scheme. To the best of our knowledge, there are only a handful of papers studying the Terwilliger algebra of conjugacy class schemes in the literature. In \cite{BannaiMunemasa95}, Bannai and Munemasa gave a complete characterization of abelian groups and showed that the corresponding schemes are triply transitive (see Section~\ref{sect:background} for the definition). In \cite{BalmacedaOura94}, Balmaceda and Oura determined the ones arising from the groups $\alt{5}$ and $\sym{5}$. More recently, Hamid and Oura \cite{HamidOura19}, determined the Terwilliger algebra of $\sym{6}$. A new resource that gives a comprehensive overview of the theory of Terwilliger algebras for conjugacy class schemes is the master's thesis  \cite{Bastian-Thesis2021} of Bastian; in particular, this thesis provides many computational results  and poses many open questions. Motivated by one of these questions in \cite{Bastian-Thesis2021} on dihedral groups, Maleki \cite{maleki2024terwilliger} determined the Terwilliger algebra of metacyclic groups that are semidirect product of a cyclic group of order $n$ by a cyclic group of order $2$. Maleki's results were recently extended to more general families of metacyclic groups \cite{yang2024terwilliger} by Feng, Yang and Zhang. 
	
    Motivated by this lack of general understanding of the Terwilliger algebra of conjugacy class schemes of groups, we start an in-depth program whose purpose is to analyze these algebras. In this paper, we solely focus on one of the next open cases for the symmetric group, that is, the group $\sym{7}$. To state our main results, we first need to recall some terminology.
	
\medskip
    	Formally, an \itbf{association scheme} is a pair $(X,\mathcal{S})$, where $X$ is a non-empty finite set and $\mathcal{S} = \{S_0,S_1,\ldots,S_d\}$ is a set of relations on $X$, satisfying the properties:
    	\begin{enumerate}[(i)]
		\item $S_0 = \{ (x,x):x\in X\}$,
		\item $\mathcal{S}$ is a partition of $X\times X$,
		\item for any $0\leq i\leq d$, $S_i^* = \{(y,x): (x,y) \in S_i\} \in \mathcal{S}$,\label{star}
		\item there exists non-negative integers $(p_{ij}^k)_{i,j,k\in \{0,1,\ldots,d\}}$ called the \itbf{intersection numbers} with the property that for any $(x,y)\in S_k$, we have 
		\begin{align*}
			p_{ij}^k = \left| \left\{ z\in X: (x,z)\in S_i,\ (z,y)\in S_j \right\} \right| .
		\end{align*}
	    \end{enumerate}
    	If $S_i^* = S_i$ for all $0\leq i\leq d$, then the association scheme is called \itbf{symmetric}. Moreover, an association scheme with intersection numbers $(p_{ij}^k)$ is \itbf{commutative} if $p_{ij}^k = p_{ji}^k$, for $0\leq i,j,k\leq d.$
	
    	Throughout this section, we let $(X,\mathcal{S})$ be an association scheme, where $\mathcal{S} = \{S_0,S_1,\ldots,S_d\}$. Then, each relation $S_i\in \mathcal{S}$ determines a directed graph on $X$ whose arc set is $S_i$. Let $A_i$ be the adjacency matrix of the digraph corresponding to $S_i$ for $1\leq i\leq d$. Using this, one can also rephrase the definition of association schemes in terms of adjacency matrices. With this reformulation, we can view association schemes from a more algebraic viewpoint, and we can define certain algebras that are associated with the scheme. The \itbf{Bose-Mesner algebra} $\mathcal{B}$ is the subalgebra of $\operatorname{M}_{|X|}(\mathbb{C})$ generated by the set of all adjacency matrices $\{A_0,A_1,\ldots,A_d\}$. This set is also a basis for the Bose-Mesner algebra since
    	\begin{align*}
		A_iA_j = \sum_{k=0}^d p_{ij}^k A_k
    	\end{align*}
    	for all $0\leq i,j\leq d$, and $\sum_{i = 0}^d A_i = J_{|X|}$ is the matrix of all-ones. The \itbf{standard module} of the association scheme $(X,\mathcal{S})$ is the subspace 
    	\begin{align*}
	 	V = \bigoplus_{x\in X} \mathbb{C} x.
    	\end{align*}
	
    	Another algebra associated with the association scheme $(X,\mathcal{S})$ is its \itbf{Terwilliger algebra} with respect to a fixed vertex of $X$. Fix an element $x\in X$. For any $0\leq i\leq d,$ define the matrix $E_i^*(x)$ to be the diagonal matrix whose diagonal is the row of the adjacency matrix $A_i$ corresponding to $x$. The \itbf{Terwilliger algebra} with respect to $x$ is the complex algebra 
    	\begin{align*}
		T_{x}(\mathcal{S}) = \langle A_0,A_1,\ldots,A_d,E_0^*(x),E_1^*(x),\ldots,E_d^*(x)\rangle.
    	\end{align*}
       As $T_x(\mathcal{S})$ is a self-adjoint subalgebra of $\operatorname{M}_{|X|}(\mathbb{C})$, it is semisimple, so it is algebraically determined by the isomorphism classes of its irreducible modules. A subspace $W$ of $V$ for which $TW\subseteq W$ is called a $\itbf{T-module}$. A $T$-module $W$ is $\emph{irreducible }$ whenever it is nonzero and has no nontrivial submodules.  Let $W$ denote an irreducible $T$-module, then it is a direct sum of non-vanishing subspaces $E^*_iW$, for $0\leq i\leq d$.

    	The $T$-module $W$ is \itbf{thin} whenever $\dim(E^*_iW)\leq1$ for $0\leq i\leq d$. From a representation theory perspective, it is of interest to know the dimension of a given Terwilliger algebra and its Wedderburn decomposition, these require the list of dimensions of the inequivalent irreducible modules. From a purely algebraic combinatorics perspective, knowing the irreducible modules of the Terwilliger algebra can help study the combinatorial properties of the underlying association scheme or objects associated with it.
	
    	In this paper, we determine some properties of the Terwilliger algebra of the conjugacy class scheme of the symmetric group $\sym{7}$. Recall that the conjugacy class scheme of a group $G$ with conjugacy classes $C_{0} = \{1\}, C_1,\ldots,C_d$ is the association scheme $\mathfrak{X}(G) := (G,\mathcal{S})$, where $\mathcal{S}$ consists of sets of the form $S_{i} = \{ (x,y) \in G\times G:x^{-1}y \in C_i  \}$ for $0\leq i\leq d$. Since the digraphs in $\mathcal{S}$ are Cayley digraphs, the left-regular representation of $G$ is a regular subgroup of automorphisms. That is, if $\lambda_g: G\to G$ such that $ x\to gx$, for  $g\in G$, then the subgroup $L_G = \left\{ \lambda_g : g\in G \right\}$ is a subgroup of the full automorphism group $\Aut{\mathcal{S}}$ of the scheme $(X,\mathcal{S})$. See Section~\ref{sect:background} for more details on the automorphism group of association schemes. Hence,	the automorphism group of any conjugacy class scheme is transitive. Consequently, every Terwilliger algebra of the conjugacy class scheme of a group $G$ is isomorphic to the one  with respect to the identity element $1\in G$. Henceforth, we only consider the Terwilliger algebra with respect to the element $1\in G$, and if there is no confusion, we use the notations $E_i^*=E_i^*(1)$ for $0\leq i\leq d$, and $T = T_1(\mathcal{S})$.
	
		For the symmetric groups $\sym{3}$ and $\sym{4}$, Bannai and Munemasa \cite{BannaiMunemasa95} showed that the dimensions of their Terwilliger algebras are respectively $11$ and  $43$. Moreover, the Terwilliger algebras of these two groups are equal to the centralizer algebra of the underlying group endowed with the action by conjugation (see Section~\ref{sect:background} for details), which results in the Wedderburn decomposition being straightforward.  The result for $\sym{3}$ was obtained using a computer search. Note however that the group $\sym{3}\cong \operatorname{D}_{6}$ belongs to the larger family of metacyclic groups admitting a cyclic subgroup of index $2$, whose Terwilliger algebras were recently determined by Maleki \cite{maleki2024terwilliger}. For $\sym{5}$, Balmaceda and Oura \cite{BalmacedaOura94} proved that the dimension of the Terwilliger algebra is $155$, and that it is not equal to the centralizer algebra of the action of $\sym{5}$ by conjugation. In \cite{BogaertsDukes14}, Bogaerts and Dukes used the Terwilliger algebras of $\sym{7}$ to study permutation codes from $\sym{7}$ with Hamming distance at least $4$. Additionally, they conjectured in \cite[Conjecture~3.5]{BogaertsDukes14} that the Terwilliger algebra of $\sym{n}$ and the centralizer algebra of the full stabilizer of the identity permutation in the automorphism group of the conjugacy scheme of $\sym{n}$ are equal, for $n\geq 6$. This conjecture was however disproved by Hamid and Oura \cite{HamidOura19}, who showed the dimension of the Terwilliger algebra for $\sym{6}$ is equal to $758$, while that of the centralizer algebra is $761$. In \cite{BalmacedaOura94,HamidOura19}, the Wedderburn decomposition of the Terwilliger algebras was also given. In this paper, we generalize these results by giving the dimension and the Wedderburn decomposition of the Terwilliger algebra of $\sym{7}$. Our main result, which we established with the assistance of computational algebra software \verb*|GAP| ~\cite{GAP4} and its association scheme package~\cite{BHL24} and \verb*|Sagemath|~\cite{sagemath} is stated as follows. 
    	\begin{theorem}
		Consider the symmetric group $\sym{7}$ and let $T$ be its Terwilliger algebra. Then, $\dim(T) = 4039$ and 
		\begin{center}
		\begin{align*}	
				T &= \operatorname{M}_{15}(\mathbb{C})
				\oplus\operatorname{M}_{15}(\mathbb{C})
				\oplus\operatorname{M}_{26}(\mathbb{C})\oplus\operatorname{M}_{2}(\mathbb{C})
				\oplus\operatorname{M}_{17}(\mathbb{C})
				\oplus\operatorname{M}_{16}(\mathbb{C})
				\oplus\operatorname{M}_{2}(\mathbb{C})
				\oplus\operatorname{M}_{21}(\mathbb{C}
				)
				\\& \hspace{0.5cm} \oplus\operatorname{M}_{15}(\mathbb{C}) \oplus \operatorname{M}_{2}(\mathbb{C}) \oplus\operatorname{M}_{19}(\mathbb{C})
				\oplus\operatorname{M}_{5}(\mathbb{C})\oplus\operatorname{M}_{13}(\mathbb{C})
				\oplus\operatorname{M}_{20}(\mathbb{C})
				\oplus\operatorname{M}_{2}(\mathbb{C})
				\oplus \operatorname{M}_{8}(\mathbb{C})\oplus\operatorname{M}_{20}(\mathbb{C}) \\& \hspace{0.5cm}
				\oplus\operatorname{M}_{9}(\mathbb{C})
				\oplus\operatorname{M}_{4}(\mathbb{C})
				\oplus\operatorname{M}_{9}(\mathbb{C})
				\oplus\operatorname{M}_{3}(\mathbb{C})
				\oplus\operatorname{M}_{3}(\mathbb{C})\oplus\operatorname{M}_{7}(\mathbb{C})\oplus\operatorname{M}_{5}(\mathbb{C})\oplus \mathbb{C}.	
		\end{align*}\label{thm:main}
		\end{center}
    	\end{theorem}
        We note that Bogaerts and Dukes in \cite{BogaertsDukes14} that the centralizer algebra for the $\sym{7}$ case of their conjecture was of dimension $4043$.


	\section{Background}\label{sect:background}
	
	Throughout this section, we let $(X,\mathcal{S})$ with $\mathcal{S} =\{S_0,S_1,\ldots,S_d\}$ be a commutative association scheme, and $V$ the corresponding standard module. We refer the reader to \cite{bannai-ito} for the basic terminologies on association schemes.


	\subsection{Automorphism groups and centralizer algebras}
	We can determine more information about the Terwilliger algebra $T_x(\mathcal{S})$ from a permutation group arising from the underlying association scheme $(X,\mathcal{S})$. An \itbf{automorphism} of the association scheme $(X,\mathcal{S})$ is a permutation of $X$ that preserves all the relations, in other words, an automorphism of every digraph of the scheme. The group of all automorphisms of the scheme $(X,\mathcal{S})$ is denoted by $\Aut{\mathcal{S}}$.  By definition, we have
	\begin{align*}
		\Aut{\mathcal{S}} = \bigcap_{i=0}^d \Aut{{\Gamma}_i},
	\end{align*}
	where $\Gamma_i$ is the digraph corresponding to $S_i$, for $0\leq i\leq d$.
	
	Given $H\leq \Aut{\mathcal{S}}$, we can endow the standard module $V$ with a $\mathbb{C}H$-module action by defining the action of $h\in H$ on $V$ by 
	\begin{align*}
		h. v = P_h v
	\end{align*}
	for $v \in V$ and where $P_h$ is the permutation matrix corresponding to $h$.

	The \itbf{centralizer algebra} of $H$ is the algebra $\End{H}{V}$ consisting of all matrices commuting with the action of $H$ as $\mathbb{C}H$-module on the standard module. In particular, the Bose-Mesner algebra $\mathcal{B}$ is contained in $\End{G}{V}$. The Terwilliger algebra is also contained in the centralizer algebra of a certain group. 
	
	Given a base vertex $x\in X$ and a subgroup $H$ of the automorphism group of $\mathcal{S}$, we will soon see that the Terwilliger algebra $T_x(\mathcal{S})$ is contained in the centralizer algebra $\End{H_x}{V}$, where $H_{x}$ is stabilizer of $x\in X$ in $H$.  Similar to the case of the Bose-Mesner algebra, it is not guaranteed that $T_x(\mathcal{S})$ is equal to one of these centralizer algebras $\End{H_x}{V}$, and to the best of our knowledge, there are no known conditions for this to be true in the literature.

 Given a digraph $\Gamma$ with adjacency matrix $A$, the automorphism group $\Aut{\Gamma}$ of $\Gamma$ consists of permutations of the vertices of $\Gamma$ that preserve the arcs and the non-arcs. It is not hard to see that by viewing $g\in \Aut{\Gamma}$ as a permutation matrix $P_g$, we have $P_gAP_g^{-1} = A$. Consequently, 
	\begin{align*}
		\Aut{\Gamma} = \left\{ P: \mbox{ $P$ is a permutation matrix and } PA = AP \right\}
	\end{align*}
	As $\Aut{\mathcal{S}}$ is the intersection of the automorphism groups of all the digraphs in the scheme, we have
	\begin{align}
		\Aut{\mathcal{S}} = \left\{ P\,\, | \mbox{ $P$ is a permutation matrix and }PA_i = A_iP \mbox{ for all } 0\leq i\leq d \right\}. 
	\end{align}
	Therefore, we obtain a natural matrix representation of the action of $\Aut{\mathcal{S}}$ on $V.$ That is, we may view the action of any $H \leq \Aut{\mathcal{S}}$ on $V$ as multiplication by a permutation matrix. The following observation immediately follows.
	
	\begin{lemma}
		Let $(X,\mathcal{S})$ be an association scheme and let $\mathcal{B}$ be its Bose-Mesner algebra. For any $H \leq \Aut{\mathcal{S}}$, we have that  $\mathcal{B}\subseteq \End{H}{V}$. \label{lem:BM-algebra}
	\end{lemma}
	
	Before proving the next lemma, we need to determine how the elements of $\Aut{\mathcal{S}}$ act on matrices. From what we have seen before,
	for any $g\in \Aut{\mathcal{S}}$ and $\mathbf{v}\in V$, $g.\mathbf{v} = P_g\mathbf{v}$, where $P_g$ is the permutation matrix of $g$. For $x\in X$, the entry of $P_g \mathbf{v}$ corresponding to the vertex $x$ is 
	\begin{align*}
		\left(P_g\mathbf{v}\right)_x = \sum_{y\in X}P_g(x,y)\mathbf{v}_y = P_g(x,\sigma^{-1}(x)) \mathbf{v}_{\sigma^{-1}(x)} = \mathbf{v}_{\sigma^{-1}(x)}.
	\end{align*}
	For any $0\leq i\leq d$, we now have that
	\begin{align*}
		\left(P_g A_i\right)_{xy} = \sum_{z\in X} P_g(x,z)A_i(z,y)
		=
		\begin{cases}
			1 & \mbox{ if there exist $z$ such that $\sigma(z) = x$ and $(z,y) \in S_i$},\\
			0 & \mbox{ otherwise.}
		\end{cases}
	\end{align*}
	Similarly, we have
	\begin{align*}
		\left(A_iP_g \right)_{xy} = \sum_{z\in X} A_i(x,z)P_g(z,y)
		=
		\begin{cases}
			1 & \mbox{ if there exist $z$ such that $\sigma(y) = z$ and $(x,z) \in S_i$},\\
			0 & \mbox{ otherwise.}
		\end{cases}
	\end{align*}

	\begin{lemma}
		Let $(X,\mathcal{S})$ be an association scheme, $H\leq \Aut{\mathcal{S}}$, and  $x\in X$. Then, the Terwilliger algebra $T_x(\mathcal{S})$ is contained in $\End{H_{x}}{V}.$\label{lem:T-algebra}
	\end{lemma}

	\begin{proof}
	 We can assume $H = H_x$. It is clear that $A_i \in \End{H}{V}$ for any $0\leq i\leq d$, so it is enough to show $E_i^* = E_i^*(x) \in \End{H}{V}$ for $0\leq i\leq d$. Let $P_h$ be the permutation matrix corresponding to $h\in H$. 
		
		We have
		\begin{align*}
			\left(P_hE_i^*\right)_{y,z} &= 
			\sum_{t\in X}P_h(y,t)E_i^*(t,z)\\
			&=
			P_h(y,z)E_{i}^*(z,z)\\
			&=
			\begin{cases}
				1 & \mbox{ if  $h(z) = y$ and $(x,z) \in S_i$}\\
				0 & \mbox{ otherwise.}
			\end{cases}
		\end{align*}
		Moreover, 
		\begin{align*}
			\left(E_i^*P_h\right)_{y,z} &= 
			\sum_{t\in X}E_i^*(y,t)P_h(t,z)\\
			&=
			E_{i}^*(y,y)P_h(y,z)\\
			&=
			\begin{cases}
				1 & \mbox{ if  $h(z) = y$ and $(x,y) \in S_i$}\\
				0 & \mbox{ otherwise.}
			\end{cases}
		\end{align*}
		Using the fact that $h\in H\leq \Aut{\mathcal{S}}$, we have $h(x,z) = (h(x),h(z)) = (x,y)$ and so $(x,z) \in S_i$ if and only if $(x,y) \in S_i$. We conclude that $P_hE_i^*(x) = E_i^*(x) P_h$. This shows that $T_{x}(\mathcal{S}) \subseteq \End{H}{V}$.
	\end{proof}

	We recall the following well-known fact.
	\begin{lemma}
	 Let $H$ be a subgroup of the automorphism group $\Aut{\mathcal{S}}$. Then, the dimension of the centralizer algebra $\End{H}{V}$ is equal to the number of orbitals of $H$ acting on $X$. \label{lem:number-of-orbitals}
	\end{lemma}
	
	Recall that the \itbf{permutation character} of a group action $M$ on a set $\Omega$ is the character $\pi: M \to \mathbb{C}$ such that $\pi(m)$ is the number of elements of $\Omega$ fixed by the group element $m\in M$. The permutation character $\pi$ is the character afforded by the permutation representation $M\to \operatorname{GL}_{|\Omega|}(\mathbb{C})$ such that $m\mapsto P_m$, where $P_m$ is the permutation matrix corresponding to $m\in M$. Since $\pi$ is a sum of irreducible characters, there exists a non-negative number $m_\chi$ such that 
	\begin{align*}
		\pi = \sum_{\chi\in \operatorname{Irr}(M)} \langle \pi,\chi \rangle \chi.
	\end{align*}
	
 Given $H \leq \Aut{\mathcal{S}}$, the centralizer algebra $\End{H}{V}$ is semisimple. Its Wedderburn decomposition can be obtained from the multiplicity of the irreducible characters in the permutation character of the action of $H$ as a permutation group.  

	\begin{lemma}\cite[Theorem~1.7.9]{sagan2013symmetric}
		Let $H \leq \Aut{\mathcal{S}}$ and $\pi$ its permutation character. Then, we have that
		\begin{align*}
			\End{H}{V} = \bigoplus_{\chi \in \operatorname{Irr}(H)} \operatorname{M}_{\langle \pi,\chi \rangle}\left(\mathbb{C}\right).
		\end{align*}\label{lem:decomposition-centralizer}
	\end{lemma}
	

	\subsection{The switching basis of the Terwilliger algebra}
	Fix a vertex $x\in X$ and consider the Terwilliger algebra $T = T_{x}(\mathcal{S})$.  For any $0\leq i\leq d$, we define 
	\begin{align}
		xS_i = \left\{ z\in X: (x,z)\in S_i \right\}.
	\end{align}	
	For any $0\leq i\leq d$, we let $E_i^* = E_i^*(x)$. The Terwilliger algebra $T$ admits a basis that is obtained by taking the so-called \itbf{switching products}. In this section, we describe in detail how to obtain such a basis. 
	
	In \cite{BannaiMunemasa95}, Bannai and Munemasa defined the subspace  $$T_0 = \operatorname{Span}_{\mathbb{C}} \left\{ E_i^*A_jE_k^*: 0\leq i,j,k\leq d \right\}$$ of the Terwilliger algebra $T$.  We will refer to the $E_i^* A_j E_k^*$'s as being the switching products of length $1$.  Whenever $T_0 = T$, then we say that the group $G$ is \itbf{triply regular}.  If $(X,\mathcal{S})$ is the conjugacy class scheme of a group $G$, then we say that $G$ is \itbf{triply transitive} whenever $T_0 = T_x(\mathcal{S}) = \End{G}{V}$, where $G$ acts on itself by conjugation. In this case, it is clear that if $G$ is triply transitive, then it is triply regular. 
	
	It is also not hard to see that if $T_0$ is an algebra, then it must be equal to $T$, and thus the group $G$ is triply regular. In fact, $T$ is the smallest algebra containing $T_0$, that is, $T = \langle T_0 \rangle$. If $T_0$ is not closed under matrix multiplication, then we define the subspace
	\begin{align*}
		T_1 = \operatorname{Span}_{\mathbb{C}} \left\{ E_{i}^*A_{j}E^*_{k}A_{l}E^*_{m}: 0\leq i,j,k,l,m \leq d\right\} 
	\end{align*}
	spanned by products of two basis elements of $T_0$;~i.e. $T_1$ is the span of the nonzero switching products of length $2$. If $T_1$ is not an algebra, then define the subspace $T_2$ spanned by three basis elements of $T_0$. In general, if $T_{\ell -1}$ is not an algebra for $\ell \geq 1$, then we define the subspace 
	\begin{align*}
		T_\ell = \operatorname{Span}_{\mathbb{C}} \left\{ \prod_{t=0}^{\ell} (E_{i_t}^*A_{j_t}E^*_{k_t}): 0\leq i_u,j_u,k_u \leq d, \ \mbox{ for }u\in \{0,1,\ldots,\ell\} \right\}
	\end{align*}
spanned by switching products of length $\ell$.  As the Terwilliger algebra $T$ is a finite-dimensional algebra, the chain of subspaces $(T_i)_{i\geq 0}$ eventually contains all the possible products of elements of $T_0$.  Hence, there will be a least index $c \ge 0$ for which $T_c = T_{c+1}$. In this case, $T_c$ is a subspace that is closed under multiplication and therefore is an algebra. Consequently, $T_c = \langle T_0\rangle = T$. We state this result as follows.
	\begin{lemma}
		The Terwilliger algebra is equal to $T_c$, where $c$ is the smallest positive integer for which $T_c = T_{c+1}$.\label{lem:max-ribbon-width}
	\end{lemma}
 We will refer to a basis of $T$ consisting of products of elements of $T_0$ of the form $(E_i^*A_jE^*_k)_{0\leq i,j,k\leq d}$ as a \itbf{switching basis} of $T$. Moreover,  the smallest index for which this chain of subspaces becomes stationary as the \itbf{switching width} of the Terwilliger algebra.


	\subsection{The representation theory of the symmetric group}
	We briefly recall the representation theory of the symmetric group $\sym{n}$ and the group $\sym{n}\times C_2.$
	
	Recall that for any $n\geq 2$, a partition of the integer $n$ is a non-increasing sequence $\left(\lambda_1,\lambda_2,\ldots,\lambda_k\right)$ of positive integers summing to $n$. If $\lambda = \left(\lambda_1,\lambda_2,\ldots,\lambda_k\right)$ is a partition of $n$, then we write $\lambda \vdash n$. The conjugacy classes and the irreducible characters of the symmetric group $\sym{n}$ are both indexed by partitions of the integer $n$.
	
	Given a permutation $\sigma \in \sym{n}$, the cycle type of $\sigma$ is the sequence consisting of the length of each cycle, in the disjoint cycle decomposition of $\sigma$, arranged in a non-increasing order. Since the conjugacy classes of $\sym{n}$ are determined by their cycle type, there is a one-to-one correspondence between the conjugacy classes of $\sym{n}$ and the partitions of $n$. Due to this correspondence, we adopt the following notation.
	
	\begin{notation}
		For any $n\geq 3$, we will denote the adjacency matrices in the conjugacy class scheme of $\sym{n}$ by $A_\mu$, for $\mu \vdash n$. Similarly, we denote the idempotents of the Terwilliger algebra by $E^*_\mu$, for $\mu \vdash n.$ 
	\end{notation}
	
	It is also well known that the irreducible representations of $\sym{n}$ are indexed by partitions of $n$. Given $\lambda\vdash n$, the irreducible $\mathbb{C}\sym{n}$-module corresponding to the partition $\lambda \vdash n$ is the \itbf{Specht module} $S^\lambda$. We will denote the character corresponding to $S^\lambda$ by $\chi^\lambda$, for $\lambda \vdash n$. For more details on the representation theory of $\sym{n}$, we redirect the reader to \cite{sagan2013symmetric}. For any permutation $\lambda,\mu\vdash n$ and a permutation $\sigma \in \sym{n}$ with cycle type $\mu$, we define
	\begin{align*}
		\chi^\lambda_\mu := \chi^\lambda(\sigma) \mbox{ and } f^\lambda:= \chi^\lambda(id).
	\end{align*}
	The degree $f^\lambda$ of the character $\chi^\lambda$, for $\lambda \vdash n$, can be computed using the Hook Length Formula \cite{sagan2013symmetric}. Moreover, the character value $\chi_\mu^\lambda$, for $\lambda,\mu \vdash n$, can be computed using the Murnaghan-Nakayama rule \cite{sagan2013symmetric}.

	Now, we turn our attention to $\sym{n} \times C_2$, where $C_2 = \{1,-1\}$ is a multiplicative group of order $2$. For any $\lambda \vdash n$, two elements $(\sigma,i)$ and $(\sigma^\prime,i^\prime)$ of $\sym{n}\times C_2$ are conjugate if there exists an element $(\tau,j)\in \sym{n}\times C_2$ such that 
	\begin{align*}
		(\sigma^\prime,i^\prime) &= (\tau,j)(\sigma,i)(\tau,j)^{-1}\\
		&=\left(\tau\sigma\tau^{-1}, jij^{-1}\right)\\
		&= \left(\tau\sigma\tau^{-1},i\right),
	\end{align*}
	which implies that $i^\prime = i$, and $\sigma$ and $\sigma^\prime$ are conjugate in $\sym{n}$. The conjugacy classes of $\sym{n}\times C_2$ can be indexed by the \itbf{signed partitions} $\lambda^+:=[\lambda,1]$ and $\lambda^-:=[\lambda,-1]$, where $\lambda\vdash n$, such that the conjugacy class corresponding to $\lambda^+$ is
	\begin{align*}
		\left\{ (\sigma,1): \sigma \in \sym{n} \mbox{ has cycle type }\lambda \right\},
	\end{align*}
	and the one corresponding to $\lambda^-$ is
	\begin{align*}
		\left\{ (\sigma,-1): \sigma \in \sym{n} \mbox{ has cycle type }\lambda \right\}.
	\end{align*}
	
	The irreducible characters of $\sym{n}\times C_2$ can be determined using the signed partitions. To see this, define $\phi:C_n\to \mathbb{C}$ such that $\phi(-1) = -1$. Then, the irreducible characters of $\sym{n}\times C_2$ are of the from $\chi^\lambda\otimes \phi$. Since the trivial character of $C_2$ is $\phi^2$, we may define 
	\begin{align*}
		\chi^{\lambda^+} = \chi^\lambda \otimes \phi^2 \mbox{ and }\chi^{\lambda^-} = \chi^\lambda \otimes \phi,
	\end{align*}
	for $\lambda \vdash n$.

	\section{The Terwilliger algebra of conjugacy class schemes}
	
	Let $G$ be a finite group with identity element $1$. For any $g\in G$, we define $g^G$ to be the conjugacy class of $G$ containing $g$. Let $\{g_0=1,g_1,\dots,g_d \}$ be a set of representatives of the distinct conjugacy classes of $G$.  The conjugacy class scheme of $G$ is the association scheme $(G,\mathcal{S})$ where $\mathcal{S} = \{ S_0,S_1,\ldots,S_d \}$, and $S_i = \left\{ (x,y)\in G\times G: x^{-1}y \in g_i^G \right\}$, for $0\leq i\leq d$. We note that $1S_i = g_i^G$ for all $0\leq i\leq d$. Note that for the conjugacy class scheme of a group $G$, the standard module $V$ can be viewed as the group algebra $\mathbb{C}G$.

	\subsection{The automorphism group} As mentioned before, the digraphs in the conjugacy class schemes are Cayley digraphs on $G$, whose connection sets are the conjugacy classes.
	
	For any $g\in G$, define the permutations $\lambda_g: G\to G$, $\rho_g: G\to G$, and $\alpha_g:G\to G$ such that for any $x\in G$
	\begin{align*}
		\begin{cases}
			\lambda_g(x)= gx,\\
			\rho_g(x)= xg^{-1},\\
			\alpha_g(x)= gxg^{-1}.
		\end{cases}
	\end{align*}
	Define the groups 
	\begin{align*}
		L_G &= \left\{ \lambda_g: g\in G \right\}\\
		R_G &= \left\{ \rho_g: g\in G \right\}\\
		\operatorname{Inn}(G) &= \left\{ \alpha_g: g\in G \right\}.
	\end{align*}
	The groups $L_G,R_G$, and $\operatorname{Inn}(G)$ are respectively the left-regular representation, the right-regular representation, and the group of inner automorphisms of $G$.
	
	Let $X_i = \operatorname{Cay}(G,g_i^G)$. By a well-known result of Sabidussi, we know that $L_G\leq \Aut{X_i}$ for any $0\leq i\leq d.$ In particular, $\Aut{\mathcal{S}}$ acts transitively on $X$ in this case. For the particular case of the conjugacy class scheme, we also claim that $R_G \leq \Aut{X_i}$ for any $0\leq i\leq d$. To see this, take two vertices $x,y\in G$ and $g\in G$. Then, by definition, we have $(x,y) \in R_i \Leftrightarrow x^{-1}y\in g_i^G$. Moreover,
	\begin{align*}
		 x^{-1}y = g^{-1}g x^{-1}yg^{-1}g = g^{-1} \left(xg^{-1}\right)^{-1}\left(yg^{-1}\right) g \in g_i^G.
	\end{align*}
	Consequently, we must have $\left(xg^{-1}\right)^{-1}\left(yg^{-1}\right) \in g_i^G$, which implies $(\rho_g(x),\rho_g(y)) \in S_i$. We conclude from this that 
	\begin{align}
		\mbox{$R_G\leq \Aut{\mathcal{S}}$.}
	\end{align}
	
	Now, define the map $\varphi: G\to G$ such that $\varphi(x) = x^{-1}$ for any $x\in G$. If every conjugacy class of $G$ is closed under inversion, then it is not hard to see that $\varphi \in \Aut{\mathcal{S}}$, and in fact $K = \left\langle L_G , R_G\right\rangle \rtimes \langle \varphi \rangle \leq \Aut{\mathcal{S}}$. This is for instance the case for the symmetric group. Note that $\left\langle L_G , R_G\right\rangle \cong \left(L_G \times R_G\right)/(L_G\cap R_G)$, and so $L_G \times R_G$ need not be a subgroup of the automorphism group of the scheme. An example where  $L_G \times R_G$ is not a subgroup of the full automorphism group is when $Z(G)$ is non-trivial. For any $g\in G$, we have $\alpha_g = \lambda_g\rho_g$ fixes the vertex $1\in G$. Moreover, the stabilizer of $1\in G$ in the subgroup $K= \left\langle L_G , R_G\right\rangle \rtimes \langle \varphi \rangle $ is 
	\begin{align}
		K_1 = \langle \operatorname{Inn}(G),\varphi \rangle = \operatorname{Inn}(G) \rtimes \langle \varphi \rangle.\label{eq:stabilizer}
	\end{align}
	
	\begin{remark}
		In Lemma~\ref{lem:T-algebra},  if $T\subset \End{H}{G}$, then $H$ need not be a subgroup of $\Aut{\mathcal{S}}$. If the action of $G$ by conjugation embeds into $\Aut{\mathcal{S}}$, then we must have $G = \operatorname{Inn}(G)$ or equivalently, $Z(G)$ is trivial. Therefore, if $Z(G)$ is non-trivial, then $G\not \hookrightarrow \Aut{\mathcal{S}}$ acting by conjugation, whereas $T \subset \End{G}{V}$.
	\end{remark}

	In the particular case of the symmetric group $\sym{n}$, we have the following result.
	\begin{theorem}\cite[Theorem~2.2]{farahat1960symmetric}
		The full automorphism group of the conjugacy class scheme of $\sym{n}$ is
		\begin{align*}
			H = \left(L_{\sym{n}} \times R_{\sym{n}}\right) \rtimes \langle \varphi \rangle,
		\end{align*}
		and the stabilizer of the identity permutation is equal to
		\begin{align*}
			\operatorname{Inn}(\sym{n}) \rtimes \langle \varphi \rangle.
		\end{align*}\label{thm:point-stabilizer}
	\end{theorem}
	We note that $\operatorname{Inn}(\sym{n}) \rtimes \langle \varphi \rangle = \operatorname{Inn}(\sym{n}) \times \langle \varphi \rangle$, since conjugation by an element and inversion commute in $\Aut{\mathcal{S}}$.

		
	\subsection{Bounds on the dimension of $T$}

	The next result follows from the fact that $T_0 \subset T \subset \End{G}{V}$, where $G$ acts on itself by conjugation. See \cite{BannaiMunemasa95} for the proof.
	\begin{proposition}
		We have that 
		\begin{align}
			\dim(T_0) = \left|\left\{(i,j,k):p_{ij}^k\neq 0 \right\}\right|\leq \dim T \leq   \dim \End{G}{V}=  \sum_{i=0}^d \frac{|G|}{|g_i^G|}.
		\end{align}\label{prop:bound}
	\end{proposition}

	\subsection{The centralizer algebra}
	 As we  saw in Lemma~\ref{lem:number-of-orbitals} and Lemma~\ref{lem:decomposition-centralizer}, certain properties of the centralizer of a group $H \leq \Aut{\mathcal{S}}$ depend entirely on the action of $H$. We give some particular examples for the conjugacy class scheme. By noting that $\End{\operatorname{Inn}(G)}{V} = \End{G}{V}$ where $G$ acts by conjugation on itself, we prove the following.
	 \begin{lemma}
	 	Recall that the conjugacy classes of $G$ are $g_0^G = 1^G = \{1\},g_1^G,\ldots,g_d^{G}$. Consider the group $G$ acting on itself by conjugation. Then, we have
	 	\begin{align*}
	 		\End{G}{V} =\End{\operatorname{Inn}(G)}{V} = \bigoplus_{\chi \in \operatorname{Irr}(G)} \operatorname{M}_{d_\chi}(\mathbb{C}),
	 	\end{align*}
	 	where 
	 	\begin{align*}
	 		d_\chi = \sum_{i=0}^d \chi(g_i)
	 	\end{align*}
	 	 is the row sum corresponding to the irreducible character $\chi$ in the character table.\label{lem:row-sum}
	 \end{lemma}
	 \begin{proof}
	 	Let $\pi$ be the permutation character of $G$ acting on itself by conjugation. By Lemma~\ref{lem:decomposition-centralizer}, we have
	 	\begin{align*}
	 		d_\chi = \langle \pi,\chi \rangle&= \langle \chi,\pi \rangle\\ &= \frac{1}{|G|} \sum_{i=0}^d \pi(g_i) \chi(g_i)|g_i^G| \\
	 		&= \frac{1}{|G|} \sum_{i=0}^d |\operatorname{C}_G(g_i)| \chi(g_i)|g_i^G|\\
	 		&= \sum_{i=0}^d \chi(g_i).
	 	\end{align*}
	 	This completes the proof.
	 \end{proof}

	 \begin{remark}
	 	An immediate consequence of Lemma~\ref{lem:row-sum} is that the row sums of the character table of any given group are non-negative.  For the symmetric group $\sym{n}$, Frumkin \cite{Frumkin1986} proved that the row sums of the character table are positive for any $n \ge 3$.  An open problem posed by Stanley \cite{Stanley2000} asks for a combinatorial proof for the non-negativity of the row sums of the character table of $\sym{n}$. 
	 \end{remark}

	\subsection{The centrally primitive idempotents}
	Given a semisimple algebra $\mathcal{A}$, we let $CPI(\mathcal{A})$ be the set of all centrally primitive idempotents of $\mathcal{A}$. 
	
	For any nonempty set $X$ and $x,y\in X$, the matrix $\mathcal{E}_{(x,y)}$ is defined to be the matrix indexed in its rows and columns by $X$, with the property that the $(u,v)$-entry of $\mathcal{E}_{(x,y)}$ is $1$ if $(u,v) = (x,y) $, and zero otherwise.  
    Let $H$ be a finite group acting on the set $X$, let $V = \mathbb{C}X$, and let $\pi$ be the corresponding permutation character.  If $\chi \in Irr(H)$ is such that $\langle \pi,\chi \rangle \neq 0$, then the primitive idempotent giving rise to the Wedderburn component of $\End{H}{V}$ is 
	$$ \epsilon_{\chi} = \frac{\chi(1)}{|H|} \sum_{h \in H} \overline{\chi(h)} M_h, \mbox{ where } M_h = \sum_{x \in X} \mathcal{E}_{(x,x^h)}.$$
	
	Now, when we take $H = \Aut{\mathcal{S}}$ to be the automorphism group of the conjugacy class scheme of the finite group $G$, since $\tilde{T} = \End{H}{V}$ is the centralizer algebra for the image of a representation of the group algebra $\mathbb{C}H_1$, it comes with a natural formula for its centrally primitive idempotents: if $R$ is any representation of $\mathbb{C}H_1$, let $$\rho_R = \sum_{\chi \in Irr(H_1)} m_{\chi} \chi$$ be the decomposition of $\rho_R$ in terms of the irreducible characters of $H_1$.  Then for each $\chi \in Irr(H_1)$, the image under $R$ of the centrally primitive idempotent 
	$$e_{\chi} = \frac{\chi(1)}{|H_1|} \sum_{h \in H_1} \chi(h^{-1}) h$$
	will be $0$ when $m_{\chi}=0$ and a centrally primitive idempotent of $R(\mathbb{C}H_1)$ when $m_{\chi} > 0$.  By a double centralizer argument, these are also the centrally primitive idempotents of $\tilde{T}$, and the dimension of the simple algebra $\tilde{T}R(e_{\chi})$ is $m_{\chi}^2$.  Since the Terwilliger algebra $T$ is a semisimple unital subalgebra of $\tilde{T}$, there is a relationship between its centrally primitive idempotents and those of $\tilde{T}$. 
	
	\begin{lemma} 
	Suppose $\mathcal{B}$ is a semisimple unital subalgebra of the semisimple algebra $\mathcal{A}$.  Let $CPI(\mathcal{A}) = \{ e_1, \dots, e_{\ell} \}$.  
	For each $e_i \in CPI(\mathcal{A})$, we have the following possibilities.   
	\begin{enumerate}[(i)]
	\item If $e_i \in \mathcal{B}$, and $e_i$ is primitive in $Z(\mathcal{B})$, then $e_i \in CPI(\mathcal{B})$ and $\mathcal{B}e_i$ is a simple subalgebra of the simple algebra $\mathcal{A}e_i$. 
		
	\item If $e_i \in \mathcal{B}$, and $e_i$ is not primitive in $Z(\mathcal{B})$, then $e_i = e'_{i,1} + \dots + e'_{i,k}$ for $e'_{i,j} \in CPI(\mathcal{B})$ with $k \ge 2$. In this case $\sum_{j=1}^k \dim(\mathcal{B} e'_{i,j}) = \dim(\mathcal{B}e_i) \le \dim \mathcal{A} e_i$. 
		
	\item If $e_i \not\in \mathcal{B}$, then there is an $e' \in CPI(\mathcal{B})$ such that $e_i e' \ne 0$.  For any such $e'$, $\mathcal{B}e' \simeq \mathcal{B}e'e_i \subseteq \mathcal{A}e_i$.
	
    Furthermore, if $F = \{e_j \in CPI(\mathcal{A}) : e_j \not\in \mathcal{B} \}$ and $K = \{ e'_k \in CPI(\mathcal{B}) : e'_k \not\in Z(\mathcal{A}) \}$, then $\sum_{e_j \in F} e_j = \sum_{e'_k \in K} e'_k$.  So, $\sum_{e'_k \in K} \dim(\mathcal{B} e'_k) \le \sum_{e_j \in F} \dim(\mathcal{A} e_j)$. \label{third}
	\end{enumerate}\label{lem:cpi}
	\end{lemma}
	
	\begin{proof}
		(i) Let $e_i \in CPI(\mathcal{A})$.  If $e_i \in \mathcal{B}$, then $e_i \in Z(\mathcal{A}) \cap \mathcal{B} \subseteq Z(\mathcal{B})$.  If $e_i$ is primitive in $Z(\mathcal{B})$ then $e_i \in CPI(\mathcal{B})$.
		
		(ii) Let $e_i \in CPI(\mathcal{A})$, and suppose $e_i \in \mathcal{B}$.  As in (i), $e_i \in Z(\mathcal{B})$.  If $e_i$ not primitive in $Z(\mathcal{B})$, it will be a sum of $k \ge 2$ idempotents in $CPI(\mathcal{B})$, which are necessarily not central in $\mathcal{A}$, and $\mathcal{B} e_i$ the direct sum of the simple components of $\mathcal{B}$ that they determine.  
		
		(iii) Let $e_i \in CPI(\mathcal{A})$.  Suppose $e_i \not\in \mathcal{B}$. Since the sum of the elements of $CPI(\mathcal{B})$ is $1$, there is at least one $e' \in CPI(\mathcal{B})$ for which $e_ie' \ne 0$.  For this $e'$, the two-sided ideal $\mathcal{A} e' \mathcal{A}$ is equal to $\oplus_j \mathcal{A}e_j$, as $e_j$ runs over $C = \{e_j \in CPI(\mathcal{A}) : e_je' \ne 0 \}$.  It follows that $f = \sum_{e_j \in C} e_j$ is the identity of $\mathcal{A}e'\mathcal{A}$, and multiplication by $f$ is the projection map from $\mathcal{A}$ to $\mathcal{A}e'\mathcal{A}$. 

        Now, the simple algebra $\mathcal{B}e'$ has identity $e'$, and embeds into $\mathcal{A}e'\mathcal{A}$.  This embedding is unital if and only if $e' = f$.  In this case,  
        $$ \mathcal{B}e' = \mathcal{B}(\sum_{e_j \in C} e_j) \hookrightarrow \oplus_{e_j \in C} \mathcal{B}e_j \subseteq \oplus_{e_j \in C} \mathcal{A} e_j, $$
        so $\mathcal{B}e'$ embeds diagonally in $\oplus_{e_j \in C} Be_j$, with $\mathcal{B}e' \simeq \mathcal{B}e_i$.      If $e' \ne f$, then 
        $$\mathcal{B} e' = \mathcal{B} e' f \subseteq \mathcal{B} f = \oplus_{e_j \in C} \mathcal{B} e_j \subseteq \oplus_{e_j \in F} \mathcal{A}e_j. $$
        Since $\mathcal{B}e'$ is simple and $e'e_i \ne 0$, we have $\mathcal{B} e' \simeq \mathcal{B} e' e_i$, which embeds into $\mathcal{A}e_i$.  

        In general, we can sort the idempotents in $CPI(\mathcal{A})$ so that 
        $$\sum_{e_j \not\in \mathcal{B}}e_j = 1 - \sum_{e_k \in \mathcal{B}} e_k \in \mathcal{B},$$
        and hence this sum will be an idempotent in $Z(\mathcal{B})$.  Therefore, it is equal to the sum of the elements of $CPI(\mathcal{B})$ that did not occur in parts (i) and (ii).  These are the elements of $CPI(\mathcal{B})$ that are not elements of $Z(\mathcal{A})$.  Since $\mathcal{B}e \subseteq \mathcal{A}e$ for any idempotent $e \in Z(\mathcal{A})$, our dimension assertion follows. 
	\end{proof}
	
	Lemma~\ref{lem:cpi} will be particularly useful when $\dim(T)$ is close to $\dim(\tilde{T})$, since it forces the irreducible representations of largest dimension to match.  
	\begin{corollary}
		If $\dim(\tilde{T}) - \dim(T)$ is less than $m_{\chi}^2 - (m_{\chi}-1)^2$, for some $e_{\chi} \in CPI(\tilde{T})$, then we must have $T e_{\chi} = \tilde{T} e_{\chi}$.\label{cor:large-dim}
	\end{corollary}

	\section{Revisiting the Terwilliger algebra of $\sym{n}$ for $3\leq n\leq 6$}\label{sect:revisit}
	Let $n\geq 3$, and consider the symmetric group $\sym{n}$. For any $\mu\vdash n$,  let $C_\mu$ be the conjugacy class whose elements have cycle type $\mu$. The character table of the conjugacy class scheme of the group $\sym{n}$ is the central character normalization of the usual character table of the group. That is, if $C$ is the character table of $\sym{n}$, then for any $\lambda,\mu \vdash n$ the $(\lambda,\mu)$-entry of $C$ is $\chi_\mu^\lambda$. If $P$ is the character table of the conjugacy class scheme of $\sym{n}$, then the $(\lambda,\mu)$-entry of $P$, for $\lambda,\mu \vdash n$, is equal to $\frac{\chi_{\mu}^\lambda|C_\mu|}{f^\lambda}$. Moreover, the multiplicity of the eigenvalue $\frac{\chi_{\mu}^\lambda|C_\mu|}{f^\lambda}$ is $\left(f^\lambda\right)^2$.
	
	Recall that the full automorphism group of the conjugacy class scheme of $\sym{n}$ is the group $H$ given in Theorem~\ref{thm:point-stabilizer}. The stabilizer of the identity element of $\sym{n}$ is denoted by $H_1$. For every $n\geq 3$, the block dimension decomposition of $\tilde{T}=\End{H_1}{V}$ is the table which is indexed in its rows and columns by partitions of $n$, and whose entries are the dimensions of the subspaces of the form $E_\mu^{*}\tilde{T}E_{\lambda}^*$, for $\lambda,\mu \vdash n$. A consequence of Lemma~\ref{lem:number-of-orbitals} is that the subspace $E_\mu^{*}\tilde{T}E_{\lambda}^*$, for $\lambda,\mu \vdash n$, has dimension equal to the number of orbitals of $H_1$, with representatives belonging to $C_\mu \times C_\lambda$.
	
	Similarly, we define the block dimension decomposition of $T$ and $T_i$ for $i\geq 0$, in the same manner. However, the dimensions cannot be computed using the orbitals of $H_1$, the same as in the case of $\tilde{T}$.

	\subsection{The case when $n=3$}
Let $\mathfrak{X}_3$ be the conjugacy class scheme of $\sym{3}$. The character table of $\sym{3}$ is equal to 

	\begin{table}[H]
		\begin{tabular}{c||ccc}
			$\sym{3}$& $\left[1^3\right]$ & $\left[2, 1\right]$ & $\left[3\right]$ \\ \hline \hline
			$\left[3\right]$ & $1$ & $1$ & $1$ \\
			$\left[2, 1\right]$ & $2$ & $0$ & $-1$ \\
			$\left[1^3\right]$ & $1$ & $-1$ & $1$ \\
		\end{tabular}
		\caption{Character table of $\sym{3}$.}
	\end{table}
	
	Let $T$ be the Terwilliger algebra of $\mathfrak{X}_{3}$, $H$ its automorphism group, $\hat{T} = \End{\sym{3}}{V}$ where $\sym{3}$ acts by conjugation, and $\tilde{T} = \End{H_{1}}{V}$, where $H_{1}$ is the stabilizer of the identity element. For this case, we have $\hat{T} = \tilde{T}$.  Further, $\dim(T)=11$, and $T = \tilde{T}$. It is not hard to show that $\dim(T_0) = 11$ by counting the non-zero intersection numbers of the association scheme. Consequently, the group $\sym{3}$ is triply transitive. Moreover, by Lemma~\ref{lem:decomposition-centralizer}, we know that the Wedderburn decomposition is determined by the permutation character of $H_{1} = \sym{3} \times C_2$ in its action on $\sym{3}$. The permutation character is given by
	\begin{align*}
		\rho_{K_{1}} = 3\chi^{[3]^+}+\chi^{[2,1]^+}+ \chi^{[1^3]^-}.
	\end{align*}
	Therefore, the Wedderburn decomposition is given by
	\begin{align*}
			T = \operatorname{M}_{3}(\mathbb{C})\oplus \mathbb{C} \oplus \mathbb{C}.
		\end{align*}
	The table containing the dimensions of $E_\mu^* T E_\lambda^*$, for $\mu,\lambda\vdash 3$, is given below.
	
    	\begin{table}[H]
			\centering
			\begin{tabular}{c||ccc}
			$T = \tilde{T}$ & $\left[1^3\right]$ & $\left[2, 1\right]$ & $\left[3\right]$ \\ \hline \hline
			$\left[1^3\right]$ & $1$ & $1$ & $1$ \\
			$\left[2, 1\right]$ & $1$ & $2$ & $1$ \\
			$\left[3\right]$ & $1$ & $1$ & $2$ \\
			\end{tabular}
			\caption{Block dimension decomposition of the Terwilliger algebra $T$ for $\sym{3}$.}
		\end{table}

	\subsection{The case when $n=4$}
	Now consider the symmetric group $\sym{4}$ and let $\mathfrak{X}_4$ be its conjugacy class scheme. Define $H = \Aut{\mathfrak{X}_4}$, and recall that $H_{1} = \sym{4} \times C_2$. Set $\tilde{T} = \End{H_{1}}{V}$, where $V= \mathbb{C}\sym{4}$ is the standard module. The character table of $\sym{4}$ is
	\begin{table}[H]
		\centering
		\begin{tabular}{c||ccccc}
			$\sym{4}$& $\left[1^4\right]$ & $\left[2, 1^2\right]$ & $\left[2^2\right]$ & $\left[3, 1\right]$ & $\left[4\right]$ \\ \hline\hline
			$\left[4\right]$ & $1$ & $1$ & $1$ & $1$ & $1$ \\
			$\left[3, 1\right]$ & $3$ & $1$ & $-1$ & $0$ & $-1$ \\
			$\left[2^2\right]$ & $2$ & $0$ & $2$ & $-1$ & $0$ \\
			$\left[2, 1^2\right]$ & $3$ & $-1$ & $-1$ & $0$ & $1$ \\
			$\left[1^4\right]$ & $1$ & $-1$ & $1$ & $1$ & $-1$ \\
		\end{tabular}
		\caption{Character table of $\sym{4}$.}
	\end{table}
	
	The group $\sym{4}$ is not triply regular, however, in the following theorem, we will see that $T = \tilde{T}$ still holds.

	\begin{theorem} 
		The following statements hold for the Terwilliger algebra of the symmetric group $\sym{4}$.
		\begin{enumerate}[(1)]
			\item $\dim(T_0) = 42$, $\dim (T) = 43$, $T=\tilde{T} = \hat{T}$.
			\item $T = \operatorname{M}_5(\mathbb{C}) \oplus\operatorname{M}_2(\mathbb{C})
			\oplus \operatorname{M}_3(\mathbb{C}) \oplus \operatorname{M}_2(\mathbb{C}) \oplus  \mathbb{C} .$
			\item $T$ has five irreducible modules of dimensions $5$ (primary), $2$, $3$, $2$, and $1$, respectively. Every irreducible $T$-module is thin.\label{thm:sym4-iii}
		\end{enumerate}\label{thm:sym4}
	\end{theorem} 

	\begin{proof}
		(1) By counting the number of non-zero intersection numbers, we deduce that $\dim(T_0) = 42$. However, $E_{[3,1]}^*T_0E_{[3,1]}^*$ is not closed under matrix multiplication, and thus not a subalgebra. Consequently, $T_0 \subsetneq T$.
		
		By Lemma~\ref{lem:row-sum}, the dimensions of the Wedderburn components of the centralizer algebra $\hat{T}$, of $\sym{4}$ acting by conjugation, are the row sums of the character table of $\sym{4}$. Hence, they are $5, 2, 3,2, \mbox{ and } 1$. We deduce that $\dim(\hat{T}) = 5^2+2^2+3^2+2^2+1^2 = 43$. Since $T\subset \hat{T}$ and $T_0 \subsetneq T$, it follows from the fact that $\dim(\hat{T})-\dim(T_0) = 1$ that $T= \tilde{T}= \hat{T}$.
		
		(2)	Similar to the case of $\sym{3}$, since $T = \tilde{T}$, we can easily determine the Wedderburn decomposition of $T$ using Lemma~\ref{lem:decomposition-centralizer}. Again, let $\rho_{H_{1}}$ be the permutation character of the permutation group $H_{1}$. Then, $\rho_{H_1}$ decomposes as 
		\begin{align*}
			\rho_{H_{1}} = 5 \chi^{[4]^+} + 2 \chi^{[3,1]^+}+ 3 \chi^{[2^2]^+} + 2 \chi^{[2,1^2]^-} + \chi^{[1^4]^-} ,
		\end{align*}
		so the irreducible characters of $\tilde{T}$ have degree $5, 2, 3, 2, 1$.  Therefore, $T = \tilde{T}$ and we obtain a formula for the centrally primitive idempotents of $T$ using the one we have automatically for $\tilde{T}$. Consequently, we have
		\begin{align*}
			T = \operatorname{M}_5(\mathbb{C}) \oplus\operatorname{M}_2(\mathbb{C})
			\oplus \operatorname{M}_3(\mathbb{C}) \oplus \operatorname{M}_2(\mathbb{C}) \oplus  \mathbb{C} .
		\end{align*}
		
		(3) Next, we show that every irreducible $T$-module is thin. Let $W_{\lambda^\pm}$ be an irreducible $\tilde{T}$-module with corresponding centrally primitive idempotent $e_{\chi^{\lambda^\pm}}$. We note that 
		\begin{align}
			E_\mu^* e_{\chi^{\lambda^\pm}} E_\mu^*\neq 0 \Longleftrightarrow E_{\mu}^* W_{\lambda^\pm} \neq (0)\label{eq:projection}
		\end{align}
		for any $\mu \vdash 4$. Moreover, we have
		\begin{align}
			\dim W_{\lambda^\pm} = \dim \left(\sum_{\mu \vdash 4} E_\mu^* \right)W_{\lambda^\pm} = \sum_{\mu \vdash 4}\dim \left( E_\mu^*W_{\lambda^\pm} \right).\label{eq:thin}
		\end{align}
		For the primary module $W_{[4]^+}$ (of dimension $5$), $E_\mu^* e_{\chi^{[4]^+}} E_\mu^* \neq 0$ for all $\mu \vdash 4$. We conclude that $\dim E_{\mu}^*W_{[4]^+} = 1$, for all $\mu \vdash 4$, which means that $W_{[4]^+}$ is thin. For the irreducible $\tilde{T}$-module $W_{[3,1]^+}$, there are exactly two partitions $\mu$ and $\mu^\prime$ of $4$, such that $E_\mu^* e_{\chi^{[3,1]^+}} E_\mu^* \neq 0$ and $E_{\mu^\prime}^* e_{\chi^{[3,1]^+}} E_{\mu^\prime}^* \neq 0$. These force $W_{[3,1]^+}$ to be thin. We argue in the same way for $W_{[2,1^2]^-}$. For $W_{[2^2]^+}$, we compute the elements of the form $E_\mu^* e_{\chi^{[2^2]^+}} E_\mu^*$, for $\mu \vdash 4$. It turns out that there are exactly three of these elements that are not equal to $0$. Since $\dim W_{[2^2]^+} = 3$, we conclude that $W_{[2^2]^+}$ is thin. The irreducible $\tilde{T}$-module $W_{[1^4]^-}$ is clearly thin since its dimension is $1$. 
		 This completes the proof.		
	\end{proof}

	In the next table, we give the dimensions of the subspaces $E_\lambda^* \tilde{T} E_\mu^*$, for $\lambda,\mu \vdash 4$. 
	\begin{table}[H]
		\centering
		\begin{tabular}{c||ccccc}
			$\tilde{T}$ & $\left[1^4\right]$ & $\left[2, 1^2\right]$ & $\left[2^2\right]$ & $\left[3, 1\right]$ & $\left[4\right]$ \\ \hline \hline
			$\left[1^4\right]$ & $1$ & $1$ & $1$ & $1$ & $1$ \\
			$\left[2, 1^2\right]$ & $1$ & $3$ & $2$ & $2$ & $2$ \\
			$\left[2^2\right]$ & $1$ & $2$ & $2$ & $1$ & $2$ \\
			$\left[3, 1\right]$ & $1$ & $2$ & $1$ & $4$ & $2$ \\
			$\left[4\right]$ & $1$ & $2$ & $2$ & $2$ & $3$ \\
		\end{tabular}
		\caption{Block dimension decomposition of the centralizer algebra $\tilde{T}$ for $\sym{4}$.}
	\end{table}

	\subsection{The case when $n = 5$}
	Consider the conjugacy class scheme $\mathfrak{X}_5$ of $\sym{5}$. Let $H=\Aut{\mathfrak{X}_5}$, and recall that the stabilizer of the identity permutation is $H_{1} = \sym{5} \times C_2$. Again, let $T$ be the Terwilliger algebra of this scheme, and $\tilde{T} = \End{H_{1}}{V}$, where $V$ is the standard module.

	\begin{theorem}
		The following statements hold for the Terwilliger algebra of the symmetric group $\sym{5}$.
		\begin{enumerate}[(1)]
			\item $\dim(T_0) = 124$, $\dim(T_1) = \dim(T) = \dim(\tilde{T}) = 155.$
			\item $T = \operatorname{M}_7(\mathbb{C})\oplus \operatorname{M}_{5}(\mathbb{C}) \oplus \operatorname{M}_{6}(\mathbb{C})\oplus \operatorname{M}_{5}(\mathbb{C}) \oplus \operatorname{M}_{3}(\mathbb{C}) 
			\oplus \mathbb{C} \oplus \operatorname{M}_{3}(\mathbb{C}) \oplus \mathbb{C} .$
			\item  $T$ has eight irreducible modules, 
			of dimensions $7$, $5$, $6$, $5$, $3$, $1$, $3$, and $1$, respectively.  All but one (of dimension $5$) are thin.    
		\end{enumerate}  
		\label{prop:dim-5}
	\end{theorem} 

		\begin{proof}
			(1) By counting the number of non-zero intersection numbers, we know that $\dim(T_0) = 124$. The subspace $E_{[4,1]}^*TE_{[4,1]}^*$ is not closed under multiplication, so it is as in the case of $n=4$, not a subalgebra. Consequently, $T_0 \subsetneq T$. Next, we compute the block dimension decomposition of $T_1$, which is given in the next table. We highlight the increase in dimension in each block by adding to $\dim(E_\lambda^* T_0 E_\mu^*)$ the number $(\dim(E_\lambda^* T_1 E_\mu^*) -\dim(E_\lambda^* T_0 E_\mu^*))$, for $\lambda,\mu  \vdash 5$.
			
			\begin{table}[H]
				\begin{tabular}{c||ccccccc}
					$T_1$& $\left[1^5\right]$ & $\left[2, 1^3\right]$ & $\left[2^2, 1\right]$ & $\left[3, 1^2\right]$ & $\left[3, 2\right]$ & $\left[4, 1\right]$ & $\left[5\right]$ \\ \hline\hline
					$\left[1^5\right]$ & $1$ & $1$ & $1$ & $1$ & $1$ & $1$ & $1$ \\
					$\left[2, 1^3\right]$ & $1$ & $3$ & $3$ & $3$ & $3$ & $3$ & $2$ \\
					$\left[2^2, 1\right]$ & $1$ & $3$ & $4$ & $3$ & $3$ & $3+1$ & $3$ \\
					$\left[3, 1^2\right]$ & $1$ & $3$ & $3$ & $4+1$ & $3+2$ & $3+2$ & $3+1$ \\
					$\left[3, 2\right]$ & $1$ & $3$ & $3$ & $3+2$ & $4+1$ & $3+2$ & $3+1$ \\
					$\left[4, 1\right]$ & $1$ & $3$ & $3+1$ & $3+2$ & $3+2$ & $4+3$ & $3+2$ \\
					$\left[5\right]$ & $1$ & $2$ & $3$ & $3+1$ & $3+1$ & $3+2$ & $4+4$ \\
				\end{tabular}
				\caption{Block dimension decomposition of the subspace $T_1$ for $\sym{5}$.}
			\end{table}
			In particular, $\dim(T_1) = 155$. 
			
			Next, we determine the dimension of $\tilde{T}$. We only need to determine the permutation character of $H_{1}$ to determine the dimension of the Wedderburn components.	The permutation character of $H_{1}$ is given by 
			\begin{align}
				\rho_{H_{1}} = 7 \chi^{[5]^+} + 5 \chi^{[4,1]^+} + 6 \chi^{[3,2]^+}+ 5 \chi^{[3,1^2]^-}+ 3 \chi^{[2^2,1]^+} + \chi^{[2^2,1]^-} + 3 \chi^{[2,1^3]^-}+ \chi^{[1^5]^-} .\label{eq:perm-char-s5}
			\end{align}
			By Lemma~\ref{lem:decomposition-centralizer}, we have
			\begin{align*}
				\dim(\tilde{T}) = 7^2 + 5^2 + 6^2 + 5^2 + 3^2 +3^2  + 1+ 1 = 155.
			\end{align*}
			Therefore, $T_1 = \tilde{T}$, and so the Terwilliger algebra $T$ must be equal to $T_1 = \tilde{T}$.
			
			(2) Knowing that $T = \tilde{T}$, using Lemma~\ref{lem:decomposition-centralizer} and \eqref{eq:perm-char-s5}, the Wedderburn decomposition of $T$ is given by
			\begin{align*}
				T = \operatorname{M}_7(\mathbb{C})\oplus \operatorname{M}_{5}(\mathbb{C}) \oplus \operatorname{M}_{6}(\mathbb{C})\oplus \operatorname{M}_{5}(\mathbb{C}) \oplus \operatorname{M}_{3}(\mathbb{C}) 
				\oplus \mathbb{C} \oplus \operatorname{M}_{3}(\mathbb{C}) \oplus \mathbb{C} .
			\end{align*}
			
			(3) Since $T = \tilde{T}$, we will consider the irreducible $\tilde{T}$-modules. Consider the irreducible $\tilde{T}$-module $W_{[3,1^2]^-}$. Computing the elements of the form $E_\mu^* e_{\chi^{[3,1^2]^-}} E_\mu^*$ for $\mu \vdash 5$, there are at least three that are equal to $0$. Using the analogue of equivalence in \eqref{eq:projection} and \eqref{eq:thin} for $\sym{5}$, and the fact that $\dim(W_{[3,1^2]^-}) = 5$, we conclude that $W_{[3,1^2]^-}$ is not thin. We use the same argument as Theorem~\ref{thm:sym4}\eqref{thm:sym4-iii} to show that the remaining irreducible $T$-modules are thin.
		\end{proof}
		
	\subsection{The case when $n = 6$}
	Now, we consider the group $\sym{6}$ with the conjugacy class scheme $\mathfrak{X}_6$ , and let $H=\Aut{\mathfrak{X}_6}$ be its automorphism group. As before, we let $T$ be the Terwilliger algebra of $\sym{6}$ and $\tilde{T} = \End{H_{1}}{V}$, where $V$ is the standard module. 
	
	\begin{theorem}
		The following statements hold for the Terwilliger algebra of the symmetric group $\sym{6}$.
		\begin{enumerate}[(1)]
			\item $\dim(T_0) = 447$, $\dim(T_1) = \dim(T) = 758$ and $ \dim(\tilde{T}) = 761.$
			\item The Wedderburn decomposition of $T$ is 
			\begin{align*}
				\begin{split}
					T &= \operatorname{M}_{11}(\mathbb{C}) \oplus \operatorname{M}_{8}(\mathbb{C})\oplus \operatorname{M}_{15}(\mathbb{C})\oplus \operatorname{M}_{4}(\mathbb{C})\oplus \operatorname{M}_{9}(\mathbb{C})\oplus \operatorname{M}_{3}(\mathbb{C})\oplus \operatorname{M}_{7}(\mathbb{C})\\
					&\hspace{0.5cm}\oplus \operatorname{M}_{6}(\mathbb{C})\oplus \operatorname{M}_{9}(\mathbb{C})\oplus \operatorname{M}_{8}(\mathbb{C})\oplus \operatorname{M}_{3}(\mathbb{C})\oplus \mathbb{C}\oplus \mathbb{C}\oplus \mathbb{C}.
				\end{split}
			\end{align*}
			\item  The irreducible $T$-modules have dimensions $11$ (primary), $15$, $9$, $9$, $8$, $8$, $7$, $6$, $4$, $3$, $3$, $1$, $1$, and $1$.  The primary module along with the six irreducible $T$-modules of smallest dimensions are thin, and the seven of largest dimension are not thin. 
		\end{enumerate} 
				\label{thm:sym-5}
			\end{theorem} 
		\begin{proof}
			(1) Similar to the argument used for $\sym{5}$, we start by counting the number of non-zero intersection numbers to compute the dimension of $T_0$. We have $\dim(T_0) = 447$. The subspace $E^*_{[5,1]}T_0E_{[5,1]}^*$ is not closed under multiplication (see Table~\ref{tab:t-6}), so we have $T_0 \subsetneq T$. Next, we determine the block dimension decomposition\footnote{ Again, we highlight the growth in dimension in each block by adding to $\dim(E_\lambda^* T_0 E_\mu^*)$ the number $(\dim(E_\lambda^* T_1 E_\mu^*) -\dim(E_\lambda^* T_0 E_\mu^*))$, for $\lambda,\mu  \vdash 6$.} of $T_1$, which is given in Table~\ref{tab:t-6}. 
			\begin{table}[H]
				
				\begin{tabular}{c||ccccccccccc}
					$T_1$& $\left[1^6\right]$ & $\left[2, 1^4\right]$ & $\left[2^2, 1^2\right]$ & $\left[2^3\right]$ & $\left[3, 1^3\right]$ & $\left[3, 2, 1\right]$ & $\left[3^2\right]$ & $\left[4, 1^2\right]$ & $\left[4, 2\right]$ & $\left[5, 1\right]$ & $\left[6\right]$ \\ \hline \hline
					$\left[1^6\right]$ & $1$ & $1$ & $1$ & $1$ & $1$ & $1$ & $1$ & $1$ & $1$ & $1$ & $1$ \\
					$\left[2, 1^4\right]$ & $1$ & $3$ & $4$ & $2$ & $3$ & $5$ & $2$ & $4$ & $4$ & $3$ & $3$ \\
					$\left[2^2, 1^2\right]$ & $1$ & $4$ & $6+2$ & $4$ & $4$ & $4+4$ & $4$ & $5+3$ & $5+3$ & $4+3$ & $4+4$ \\
					$\left[2^3\right]$ & $1$ & $2$ & $4$ & $3$ & $2$ & $3$ & $3$ & $4$ & $4$ & $3$ & $5$ \\
					$\left[3, 1^3\right]$ & $1$ & $3$ & $4$ & $2$ & $5+1$ & $5+4$ & $4$ & $4+2$ & $4+2$ & $5+3$ & $4+2$ \\
					$\left[3, 2, 1\right]$ & $1$ & $5$ & $4+4$ & $3$ & $5+4$ & $6+13$ & $4+2$ & $5+8$ & $5+8$ & $5+11$ & $5+7$ \\
					$\left[3^2\right]$ & $1$ & $2$ & $4$ & $3$ & $4$ & $4+2$ & $5+1$ & $4+2$ & $4+2$ & $5+3$ & $5+4$ \\
					$\left[4, 1^2\right]$ & $1$ & $4$ & $5+3$ & $4$ & $4+2$ & $5+8$ & $4+2$ & $6+6$ & $5+7$ & $5+8$ & $5+8$ \\
					$\left[4, 2\right]$ & $1$ & $4$ & $5+3$ & $4$ & $4+2$ & $5+8$ & $4+2$ & $5+7$ & $6+6$ & $5+8$ & $5+8$ \\
					$\left[5, 1\right]$ & $1$ & $3$ & $5+2$ & $3$ & $5+3$ & $5+11$ & $5+3$ & $5+8$ & $5+8$ & $6+17$ & $5+11$ \\
					
					$\left[6\right]$ & $1$ & $3$ & $4+4$ & $5$ & $4+2$ & $5+7$ & $5+4$ & $5+8$ & $5+8$ & $5+11$ & $6+13$ \\
				\end{tabular}
				\caption{Block dimension decomposition for the Terwilliger algebra $T$ of $\sym{6}$.}\label{tab:t-6}		
			\end{table} 
			
Since $T_1$ is closed under left and right multiplication by elements of $T_0$, it is a subalgebra of $T$, which implies that $T_1 = T$. By summing the dimensions given in Table~\ref{tab:t-6}, we have $\dim(T) = 758.$
			
			For $\tilde{T}$, as we have seen before, the dimension is equal to the number of orbitals of the action of the point-stabilizer $H_1$. The block decomposition of $\tilde{T}$ is given in Table~\ref{tab:tilde-6}.
			\begin{table}[H]
				
				\begin{tabular}{c||ccccccccccc}
					$\tilde{T}$ & $\left[1^6\right]$ & $\left[2, 1^4\right]$ & $\left[2^2, 1^2\right]$ & $\left[2^3\right]$ & $\left[3, 1^3\right]$ & $\left[3, 2, 1\right]$ & $\left[3^2\right]$ & $\left[4, 1^2\right]$ & $\left[4, 2\right]$ & $\left[5, 1\right]$ & $\left[6\right]$ \\ \hline \hline
					$\left[1^6\right]$ & $1$ & $1$ & $1$ & $1$ & $1$ & $1$ & $1$ & $1$ & $1$ & $1$ & $1$ \\
					$\left[2, 1^4\right]$ & $1$ & $3$ & $4$ & $2$ & $3$ & $5$ & $2$ & $4$ & $4$ & $3$ & $3$ \\
					$\left[2^2, 1^2\right]$ & $1$ & $4$ & $8$ & $4$ & $4$ & $8$ & $4$ & $8$ & $8$ & $7$ & $8$ \\
					$\left[2^3\right]$ & $1$ & $2$ & $4$ & $3$ & $2$ & $3$ & $3$ & $4$ & $4$ & $3$ & $5$ \\
					$\left[3, 1^3\right]$ & $1$ & $3$ & $4$ & $2$ & $6$ & $9$ & $4$ & $6$ & $6$ & $8$ & $6$ \\
					$\left[3, 2, 1\right]$ & $1$ & $5$ & $8$ & $3$ & $9$ & $20$ & $6$ & $13$ & $13$ & $16$ & $12$ \\
					$\left[3^2\right]$ & $1$ & $2$ & $4$ & $3$ & $4$ & $6$ & $6$ & $6$ & $6$ & $8$ & $9$ \\
					$\left[4, 1^2\right]$ & $1$ & $4$ & $8$ & $4$ & $6$ & $13$ & $6$ & $12$ & $12$ & $13$ & $13$ \\
					$\left[4, 2\right]$ & $1$ & $4$ & $8$ & $4$ & $6$ & $13$ & $6$ & $12$ & $12$ & $13$ & $13$ \\
					$\left[5, 1\right]$ & $1$ & $3$ & $7$ & $3$ & $8$ & $16$ & $8$ & $13$ & $13$ & $24$ & $16$ \\
					$\left[6\right]$ & $1$ & $3$ & $8$ & $5$ & $6$ & $12$ & $9$ & $13$ & $13$ & $16$ & $20$ \\
				\end{tabular}
				\caption{Block dimension decomposition for the centralizer algebra $\tilde{T}$ of $\sym{6}$.}\label{tab:tilde-6}
			\end{table}
			From Table~\ref{tab:tilde-6}, we deduce that $\dim(\tilde{T}) = 761.$ 
			
			(2) Next, we determine the Wedderburn decomposition of $T$. Since $T\subsetneq \tilde{T}$, we cannot use the same argument as before. We can still obtain most of the Wedderburn components from that of $\tilde{T}$ using Corollary~\ref{cor:large-dim}. The permutation character for the action of $H_1$ on $\sym{6}$ is given by
			\begin{align} 
				\begin{split}
					\rho_{H_{1}} &= 11 \chi^{[6]^+} + 8 \chi^{[5,1]^+}+ \chi^{[5,1]^-} + 15 \chi^{[4,2]^+}+ 4 \chi^{[4,2]^-}+ \chi^{[4,1^2]^+}+ 9 \chi^{[4,1^2]^-} \\
					& \hspace{0.5cm}+ 3 \chi^{[3^2]^+}+ 7 \chi^{[3,2,1]^+}  + 6 \chi^{[3,2,1]^-}     +  \chi^{[3,1^3]^+}+ 9 \chi^{[3,1^3]^-}+ 8\chi^{[2^3]^+}+ \chi^{[2^3]^-}\\
					&\hspace{0.5cm}+ \chi^{[2^2,1^2]^+} + 3 \chi^{[2,1^4]^+}    + \chi^{[1^6]^+}.
				\end{split} \label{eq:perm-char-s6}
			\end{align}
			Denote by $m_{\chi}$ the multiplicity of the character $\chi$ in \eqref{eq:perm-char-s6}. By Corollary~\ref{cor:large-dim}, since $\dim(\tilde{T}) -\dim(T) = 3$, any irreducible character $\chi$ in \eqref{eq:perm-char-s6} such that $m_\chi^2 - (m_\chi-1)^2>3$ gives rise to a Wedderburn component of $T$ of the same dimension as the one in $\tilde{T}$. Clearly, any irreducible character in \eqref{eq:perm-char-s6} with $m_\chi\geq 3$ yields such Wedderburn components. For the remaining components, we only need to consider the restrictions of the irreducible characters 
			\begin{align}
				\chi^{[5,1]^-},\chi^{[4,1^2]^+}, \chi^{[3,1^3]^+},\chi^{[2^3]^-},\chi^{[2^2,1^2]^+},\mbox{ and }\chi^{[1^6]^+}.\label{eq:six}
			\end{align}
			There are $17$ irreducible characters involved in \eqref{eq:perm-char-s6}, $11$ of which are also in $T$. As the dimension of $Z(T)$ is $14$, the $6$ characters in \eqref{eq:six} must produce the $3$ additional irreducible characters of $T$. The centrally primitive idempotents $ \tilde{e}_{[5,1]^-},\tilde{e}_{[4,1^2]^+},\tilde{e}_{[3,1^3]^+}, \tilde{e}_{[2^3]^-},\tilde{e}_{[2^2,1^2]^+} $, and $\tilde{e}_{[1^6]^+}$ of $\tilde{T}$ do not lie in $T$, so we can use Lemma~\ref{lem:cpi}\eqref{third} to determine their restrictions into $T$. The last three centrally primitive idempotents of $\tilde{T}$ arise from the pairing of the centrally primitive idempotents of $\tilde{T}$ given by  
			\begin{align}
				  \mbox{$(\tilde{e}_{[1^6]^+},\tilde{e}_{[2^2,1^2]^+}),$ $(\tilde{e}_{[2^3]^-},\tilde{e}_{[3,1^3]^+}),$ and $(\tilde{e}_{[4,1^2]^+},\tilde{e}_{[5,1]^-})$.}\label{eq:cpi-6}
			\end{align}
			The Wedderburn components obtained from pairing these centrally primitive idempotents are of dimension equal to $\dim(\tilde{T}\tilde{e}_{[1^6]^+}) = 1$, $\dim(\tilde{T}\tilde{e}_{[2^3]^-})=1$, and $\dim(\tilde{T}\tilde{e}_{[4,1^2]^+})=1$.
			Consequently, the Wedderburn decomposition is
			\begin{align*}
				\begin{split}
					T &= \operatorname{M}_{11}(\mathbb{C}) \oplus \operatorname{M}_{8}(\mathbb{C})\oplus \operatorname{M}_{15}(\mathbb{C})\oplus \operatorname{M}_{4}(\mathbb{C})\oplus \operatorname{M}_{9}(\mathbb{C})\oplus \operatorname{M}_{3}(\mathbb{C})\oplus \operatorname{M}_{7}(\mathbb{C})\\
					&\hspace{0.5cm}\oplus \operatorname{M}_{6}(\mathbb{C})\oplus \operatorname{M}_{9}(\mathbb{C})\oplus \operatorname{M}_{8}(\mathbb{C})\oplus \operatorname{M}_{3}(\mathbb{C})\oplus \mathbb{C}\oplus \mathbb{C}\oplus \mathbb{C}.
				\end{split}
			\end{align*}
			
			(3) The primary module is clearly thin. The irreducible $T$-modules arising from the centrally primitive idempotents of $T$ given in \eqref{eq:cpi-6} are all $1$-dimensional. Hence, they are all thin. The remaining irreducible $T$-modules are restrictions of irreducible $\tilde{T}$-modules. We can apply the same argument as in Theorem~\ref{thm:sym4} and Theorem~\ref{thm:sym-5} to determine which irreducible $\tilde{T}$-modules are thin and which ones are not. Hence, we omit the computation.
			
			This completes the proof.
		\end{proof}

The fact that our $H_1 \subseteq \sym{n}$ as a permutation group on $V$ implies a general fact about the pair of irreducible $\tilde{T}$-modules corresponding to $\lambda^+$ and $\lambda^-$ for a fixed $\lambda \vdash n$.   

\begin{proposition} Let $n \ge 3$.  For any partition $\lambda \vdash n$, the sum of the dimensions of the $\tilde{T}$-modules corresponding to $\lambda^+$ and $\lambda^-$ (including the case when one of these is the zero module) is equal to the corresponding row sum $d_{\chi^{\lambda}}$ in the character table of $\sym{n}$.  
\end{proposition} 

\begin{proof} 
Let $V$ be the standard module.  Since $H_1$ contains the subgroup $\sym{n}$ acting as inner automorphisms on the group basis of $V = \mathbb{C}\sym{n}$, we have that $\tilde{T} = \End{H_1}{V} \subseteq \End{\sym{n}}{V}$.  For the same reason, the irreducible $\End{\sym{n}}{V}$-module corresponding to any $\lambda \vdash n$ will restrict to the sum of the $\tilde{T}$-modules corresponding to $\lambda^+$ and $\lambda^-$, which are either irreducible or zero.  The proposition follows.   
\end{proof} 

By the result of Frumkin, we can conclude that at least one of the two $\tilde{T}$-modules corresponding to $\lambda^+$ and $\lambda^-$ is not zero.  For the partition $[n]$, the $\tilde{T}$-module corresponding to $[n]^+$ always restricts to the primary module of $T$, so its dimension is $p(n) = d_{\chi^{[n]}}$.  So, the $\tilde{T}$-module corresponding to $[n]^-$ is always the zero module. 
		
	\section{Proof of Theorem~\ref{thm:main}}
	Consider the group $\sym{7}$ with the conjugacy class scheme $\mathfrak{X}_7$ , and let $H=\Aut{\mathfrak{X}_7}$ be its automorphism group. As before, we let $T$ be the Terwilliger algebra of $\sym{7}$ and $\tilde{T} = \End{H_{1}}{V}$, where $V$ is the standard module.
	
	First, we consider the centralizer algebra $\tilde{T}$. 
	The degrees of the irreducible representations of the centralizer algebra $\tilde{T}$ were found by Bogaerts and Dukes.  The decomposition of the permutation character for the action of $H_1$ on $\sym{7}$ is: 
	
	\begin{align}
		\begin{split}
			\rho_{K_1} &= 15\chi^{[7]^+}+  15 \chi^{[6,1]^+} + 26\chi^{[5,2]^+} + 2\chi^{[5,1^2]^+} + 17\chi^{[5,1^2]^-} + 16\chi^{[4,3]^+} +2\chi^{[4,3]^-} \\
			& \hspace{0.5cm}  + 21 \chi^{[4,2,1]^+} + 15\chi^{[4,2,1]^-} + 2\chi^{ [4,1^3]^+} + 19\chi^{[4,1^3]^-} + 5\chi^{[3^2,1]^+} + 13\chi^{[3^2,1]^-} \\
			& \hspace{0.5cm}  + 20\chi^{[3,2^2]^+} + 2\chi^{[3,2^2]^-} + 8\chi^{[3,2,1^2]^+} + 20\chi^{ [3,2,1^2]^-} + 9\chi^{[3,1^4]^+} + 4\chi^{[3,1^4]^-} \\
			& \hspace{0.5cm}  + 9\chi^{[2^3,1]^+} + 3\chi^{[2^3,1]^-} + 3\chi^{[2^2,1^3]^+} + 7\chi^{[2^2,1^3]^-} + 5\chi^{[2,1^5]^+} + \chi^{[1^7]^-}.
		\end{split}\label{eq:perm-char-7}
	\end{align}
	We conclude from this decomposition that $\dim(\tilde{T}) = 4043.$

	Now, we explain the strategy of the proof of Theorem~\ref{thm:main}. As explained in Lemma~\ref{lem:max-ribbon-width}, to find the Terwilliger algebra, one just needs to find the smallest index for which the chain of subspaces $(T_i)_{i\geq 0}$ is stationary. It is clear that we can do this by looking at the block dimension decomposition of a subspace $T_i$, and checking whether some entries are different from the ones in $T_{i-1}$.
	
	Since the conjugacy classes of $\sym{n}$ are determined by partitions of $n$, we let $A_\lambda$, for $\lambda \vdash n$, be the adjacency matrix of the Cayley graph of $\sym{n}$ with connection set equal to the conjugacy class corresponding to $\lambda$.
	
	 First, we determine the dimension of $T_0$, by counting the number of non-zero intersection numbers of $\mathfrak{X}_7$. From this, we have $\dim(T_0) = 1232$, and the block $E^*_{[6,1]}T_0E^*_{[6,1]}$ is not closed under multiplication, thus $T_0 \subsetneq T_1$. 
	 
	 Next, we determine the block dimension decomposition of $T_1$ which is given in Table~\ref{tab:T1}. From this, we know that $\dim(T_1) = 4036$. However, ${T}_1$ is not equal to $T$ since $82 = \dim(E_{[6,1]}^*T_1E_{[6,1]}^*)<\dim(E_{[6,1]}^*T_2E_{[6,1]}^*) = 83$, which implies that $T_1\ \subsetneq T_2$. Since $\dim(\tilde{T}) = 4043$, $\dim(T_1) = 4036$, and $T _1\subsetneq T_2\subset T \subset \tilde{T}$, there can only be small gaps between the dimensions of $E_\lambda^*T_1E_{\mu}^*$ and $E_\lambda^*T_2E_{\mu}^*$, for any $\lambda,\mu\vdash 7$. In fact, the only blocks where these two dimensions differ are in the blocks corresponding to the partitions $[6,1]$ and $[7]$. The part of the block dimension decomposition of $T_2$ corresponding to these two blocks is given in Table~\ref{tab:dim-T-2-sym7}.
	 \begin{table}[H]
	 	\centering
	 	\begin{tabular}{c||cc}
	 		$T_2$ & $[6,1]$ & $[7]$\\ \hline\hline
	 		$[6,1]$&$82+1$& $70+1$\\
	 		$[7]$& $70+1$&$77$\\
	 	\end{tabular}
	 	\caption{Dimension growth between $T_1$ and $T_2$.}\label{tab:dim-T-2-sym7}
	 \end{table}
	Consequently, $\dim(T_2) = 4039$. Now, we claim that $T_2 = T_3$. By contradiction, assume that $T_2 \subsetneq T_3$, and let $X \in T_3\setminus T_2$. Then, by definition, we have
	\begin{align*}
			T_3 = \operatorname{Span}_{\mathbb{C}} \Sigma
	\end{align*}
	where 
	\begin{align*}
		\Sigma = \left\{ \prod_{t=0}^{3} (E_{\lambda_t}^*A_{\alpha_t}E^*_{\mu_t}): \lambda_u,\alpha_u,\mu_u \vdash 7, \ \mbox{ for }u\in \{0,1,2,3\} \right\}.
	\end{align*}
	Since $X\in T_3\setminus T_2$, any expression of $X$ as a linear combination of elements of $\Sigma$ must contain a scalar multiple of an element of the form
	\begin{align*}
		\prod_{t=0}^{3} (E_{\lambda_t}^*A_{\alpha_t}E^*_{\mu_t}) \in T_3\setminus T_2
	\end{align*}
	where $\lambda_u,\alpha_u,\mu_u \vdash 7, \ \mbox{ for }u\in \{0,1,2,3\}$,
	otherwise, $X$ would be contained in $T_2.$ 
	
	Given  $\lambda_u,\alpha_u,\mu_u \vdash 7, \ \mbox{ for }u\in \{0,1,2,3\}$, if
	\begin{align}
		\prod_{t=0}^{3} (E_{\lambda_t}^*A_{\alpha_t}E^*_{\mu_t}) =(E_{\lambda_0}^*A_{\alpha_0}E^*_{\mu_0})(E_{\lambda_1}^*A_{\alpha_1}E^*_{\mu_1})(E_{\lambda_2}^*A_{\alpha_2}E^*_{\mu_2})(E_{\lambda_3}^*A_{\alpha_3}E^*_{\mu_3})\in T_3\setminus T_2,\label{eq:prod-4}
	\end{align}
	then 
	\begin{align*}
		(E_{\lambda_1}^*A_{\alpha_1}E^*_{\mu_1})(E_{\lambda_2}^*A_{\alpha_2}E^*_{\mu_2})(E_{\lambda_3}^*A_{\alpha_3}E^*_{\mu_3}) \in T_2\setminus T_1,
	\end{align*}
	otherwise \eqref{eq:prod-4} would belong to $T_2$. Using Table~\ref{tab:dim-T-2-sym7}, it is clear that
	\begin{align*}
		\left(\lambda_1,\mu_3\right) \in \left\{\left([6,1],[6,1]\right),\left([6,1],[7]\right),\left([7],[6,1]\right)\right\}.
	\end{align*}
	Now, we consider these three possibilities for $\left(\lambda_1,\mu_3\right)$.
	
	Assume that $\left(\lambda_1,\mu_3\right) = \left([6,1],[6,1]\right)$. Then, \eqref{eq:prod-4} becomes
	\begin{align}
		(E_{\lambda_0}^*A_{\alpha_0}E^*_{[6,1]})(E_{[6,1]}^*A_{\alpha_1}E^*_{\mu_1})(E_{\lambda_2}^*A_{\alpha_2}E^*_{\mu_2})(E_{\lambda_3}^*A_{\alpha_3}E^*_{[6,1]}).\label{eq:prod-4-2}
	\end{align}
	However, for all $\lambda_0,\alpha_0 \vdash 7$, we have checked that any element of the form \eqref{eq:prod-4-2} belongs to $T_2$. The latter is a contradiction to the element in \eqref{eq:prod-4} belonging to $T_3\setminus T_2$.
	
	For the cases $\left(\lambda_1,\mu_3\right) =  \left([6,1],[7]\right)$ and $\left(\lambda_1,\mu_3\right) = \left([7],[6,1]\right)$, we use a similar argument to get a contraction on the element in \eqref{eq:prod-4} being in $T_3\setminus T_2$.

	From these cases, we deduce the element $X \in T_3\setminus T_2$ does not exist.
	We conclude that $T_2 = T_3 = T$, and therefore, $\dim(T) = 4039.$
	
	Finally, we determine the Wedderburn decomposition of $T$. Similar to the case for $\sym{6}$, a Wedderburn component of $\tilde{T}$ corresponding to irreducible characters in \eqref{eq:perm-char-7} with multiplicity larger than $2$ also determines a Wedderburn component of $T$. Consequently, the remaining Wedderburn components of $T$ come from a combination of centrally primitive idempotents of $\tilde{T}$ corresponding to the following partitions
	\begin{align}
		\chi^{[5,1^2]^+},\chi^{[4,3]^-}, \chi^{[4,1^3]^+},\chi^{[3,2^2]^-},\mbox{ and }\chi^{[1^7]^-}.
	\end{align}
	The centrally primitive idempotents $\tilde{e}_{\chi^{[5,1^2]^+}},\tilde{e}_{\chi^{[4,3]^-}},$ and $\tilde{e}_{\chi^{[1^7]^-}}$ of $\tilde{T}$ lie in $T$, whereas $\tilde{e}_{\chi^{[4,1^3]^+}}$ and $\tilde{e}_{\chi^{[3,2^2]^-}}$ do not. By Lemma~\ref{lem:cpi}, depending on whether they are central, $\tilde{e}_{\chi^{[5,1^2]^+}},\tilde{e}_{\chi^{[4,3]^-}},$ and $\tilde{e}_{\chi^{[1^7]^-}}$ are centrally primitive idempotents of $T$ or they decompose as sums of centrally primitive idempotent of $T$. The idempotents $\tilde{e}_{\chi^{[5,1^2]^+}},\tilde{e}_{\chi^{[4,3]^-}},$ and $\tilde{e}_{\chi^{[1^7]^-}}$ correspond to irreducible $T$-modules of dimensions $2,2$, and $1$, respectively. Therefore, their corresponding Wedderburn components are $\operatorname{M}_2(\mathbb{C}), \operatorname{M}_2(\mathbb{C})$, and $\mathbb{C}$, respectively.
	
	As the two centrally primitive idempotents $\tilde{e}_{\chi^{[4,1^3]^+}}$ and $\tilde{e}_{\chi^{[3,2^2]^-}}$ of $\tilde{T}$ do not belong to $T$, we use Lemma~\ref{lem:cpi}\eqref{third}.  When we do this, we find that the sum $e$ of $\tilde{e}_{\chi^{[4,1^3]^+}}$ and $\tilde{e}_{\chi^{[3,2^2]^-}}$ lies in $T$.  Therefore, $Te$ embeds diagonally in $\tilde{T}\tilde{e}_{\chi^{[4,1^3]^+}} \oplus \tilde{T}\tilde{e}_{\chi^{[3,2^2]^-}}$, and this accounts for the entire difference in dimensions between $T$ and $\tilde{T}$.   So, this sum will be a centrally primitive idempotent of $T$ whose corresponding Wedderburn component is of dimension equal to $\dim(\tilde{T}\tilde{e}_{\chi^{[3,2^2]^-}}) = \dim(\tilde{T}\tilde{e}_{\chi^{[3,2^2]^-}}) = 2^2 = 4$. We conclude that the Wedderburn decomposition of $T$ is 
	\begin{align*}	
		T &= \operatorname{M}_{15}(\mathbb{C})
		\oplus\operatorname{M}_{15}(\mathbb{C})
		\oplus\operatorname{M}_{26}(\mathbb{C})\oplus\operatorname{M}_{2}(\mathbb{C})
		\oplus\operatorname{M}_{17}(\mathbb{C})
		\oplus\operatorname{M}_{16}(\mathbb{C})
		\oplus\operatorname{M}_{2}(\mathbb{C})
		\oplus\operatorname{M}_{21}(\mathbb{C}
		)
		\\& \hspace{0.5cm} \oplus\operatorname{M}_{15}(\mathbb{C}) \oplus \operatorname{M}_{2}(\mathbb{C}) \oplus\operatorname{M}_{19}(\mathbb{C})
		\oplus\operatorname{M}_{5}(\mathbb{C})\oplus\operatorname{M}_{13}(\mathbb{C})
		\oplus\operatorname{M}_{20}(\mathbb{C})
		\oplus\operatorname{M}_{2}(\mathbb{C})
		\oplus \operatorname{M}_{8}(\mathbb{C})\oplus\operatorname{M}_{20}(\mathbb{C}) \\& \hspace{0.5cm}
		\oplus\operatorname{M}_{9}(\mathbb{C})
		\oplus\operatorname{M}_{4}(\mathbb{C})
		\oplus\operatorname{M}_{9}(\mathbb{C})
		\oplus\operatorname{M}_{3}(\mathbb{C})
		\oplus\operatorname{M}_{3}(\mathbb{C})\oplus\operatorname{M}_{7}(\mathbb{C})\oplus\operatorname{M}_{5}(\mathbb{C})\oplus \mathbb{C}.	
	\end{align*}

\begin{remark} As we did for the smaller symmetric groups, we can determine the thin irreducible $T$-(and $\tilde{T}$-)modules.  In the case of $\sym{7}$, all of the irreducible $T$-modules of dimension $\le 5$ and the two $15$-dimensional irreducible $T$-modules corresponding to $[7]^+$ and $[6,1]^+$ turn out to be thin, and all others are not thin.   
\end{remark}

\section{Conclusion}
In this paper, we determined the Terwilliger algebra of the conjugacy class association scheme of $\sym{7}$. In particular, we found its dimension and its Wedderburn decomposition. We also revisited some of the known results for $3\leq n\leq 6$ using some newer methods.  

The case for $n\geq 8$ is still open. However, we might be able to tackle this problem by refining some of the methods in this paper. Based on what we have seen in these small cases, it seems that one of the blocks where the dimensions of $T$ and $\tilde{T}$ differ is at the entry $[n-1]\times [n-1]$. If this observation holds for all $n\geq 8$, then this is enough to show that $T\subsetneq \tilde{T}$, and this would entirely disprove the conjecture of Bogaerts and Dukes for all $n\geq 8$. We make the following conjecture.

\begin{conjecture}
	For any $n\geq 6$, we have
	\begin{align*}
		\dim(E^*_{[n-1,1]}TE^*_{[n-1,1]}) < \dim(E^*_{[n-1,1]}\tilde{T}E^*_{[n-1,1]}).
	\end{align*} 
	In particular, $T\subsetneq \tilde{T}.$
\end{conjecture}


		\appendix
		\newpage
		
		\pagestyle{empty}
		\begin{landscape}
			\section{Block dimensions for $\sym{7}$}
			\vspace*{3cm}
			\begin{table}[H]
				\centering
				{
					\begin{tabular}{c||ccccccccccccccc}
						${T}_1$& $\left[1^7\right]$ & $\left[2, 1^5\right]$ & $\left[2, 2, 1^3\right]$ & $\left[2^3, 1\right]$ & $\left[3, 1^4\right]$ & $\left[3, 2, 1^2\right]$ & $\left[3, 2^2\right]$ & $\left[3^2, 1\right]$ & $\left[4, 1^3\right]$ & $\left[4, 2, 1\right]$ & $\left[4, 3\right]$ & $\left[5, 1^2\right]$ & $\left[5, 2\right]$ & $\left[6, 1\right]$ & $\left[7\right]$ \\ \hline \hline
						$\left[1^7\right]$ & 1& 1 & 1 & 1 & 1  & 1 & 1 & 1 & 1 & 1  & 1 & 1 & 1 & 1 & 1 \\ 
						$\left[2, 1^5\right]$ & 1 & 3& 4 & 3 & 3  & 6 & 4 & 3 & 4 & 6  & 4 & 4 & 4 & 4 & 3 \\
						$\left[2^2, 1^3\right]$ & 1 & 4 &  9 & 7 & 5  &14 & 9 & 7 & 9 &16  &10 &11 &11 &14 &12 \\
						$\left[2^3, 1\right]$ & 1 & 3 & 7 & 7 & 4  &11 & 7 & 7 & 7 &14  & 9 &10 &10 &14 &11 \\
						$\left[3, 1^4\right]$ & 1 & 3 & 5 & 4 & 6  &11 & 7 & 7 & 7 &11  & 9 & 9 & 9 &11 & 8 \\
						$\left[3, 2, 1^2\right]$ & 1 & 6 &14 &11 &11  &38&21 &20 &20 &44  &28 &32 &32 &44 &36 \\
						$\left[3, 2^2\right]$ & 1 & 4 & 9 & 7 & 7  &21 &14 &11 &12 &24  &18 &19 &19 &25 &24 \\
						$\left[3^2, 1\right]$ & 	1 & 3 & 7 & 7 & 7  &20 &11 & 18&12 &25  &21 &22 &22 &33 &24 \\
						$\left[4, 1^3\right]$ & 1 & 4 & 9 & 7 & 7  &20 &12 &12 &13 &24  &17 &19 &19 &25 &21 \\
						$\left[4, 2, 1\right]$ & 1 & 6 &16 &14 &11  &44 &24 &25 &24 & 58  &36 &42 &42 &62 &51 \\
						$\left[4, 3\right]$ & 1 & 4 &10 & 9 & 9  &28 &18 &21 &17 &36  & 32 &33 &33 &47 &42 \\
						$\left[5, 1^2\right]$ & 1 & 4 &11 &10 & 9  &32 &19 &22 &19 &42  &33 &40 &40 &54 &48 \\
						$\left[5, 2\right]$ & 1 & 4 &11 &10 & 9  &32 &19 &22 &19 &42  &33 &40 & 40 &54 &48 \\
						$\left[6, 1\right]$ & 1 & 4 &14 &14 &11  &44 &25 &33 &25 &62  &47 &54 &54 &82 &70 \\
						$\left[7\right]$ & 1 & 3 &12 &11 & 8  &36 &24 &24 &21 &51  &42 &48 &48 &70 &77
					\end{tabular}
				}\caption{Block dimension decomposition of $T_1$ for $\sym{7}$}\label{tab:T1}
			\end{table}
	    	\end{landscape}

	\end{document}